\documentclass[12pt]{amsart}
\usepackage{amsmath}
\usepackage{amsfonts}
\usepackage{amssymb}
\usepackage{mathtools}
\usepackage{enumitem}
\usepackage[alphabetic]{amsrefs}

\usepackage[usenames,dvipsnames]{xcolor}

\usepackage[english]{babel}
\usepackage[latin1]{inputenc}
\usepackage{fancyhdr}
\usepackage{indentfirst}
\usepackage{graphicx}
\usepackage{float}
\usepackage{caption}
\usepackage{newlfont} 
\usepackage{amssymb}
\usepackage{amsmath}
\usepackage{mathtools}
\usepackage{latexsym}
\usepackage{amsthm}
\usepackage{enumitem}
\usepackage{esint}
\theoremstyle{plain}
\newtheorem{thm}{Theorem}[section]
\newtheorem{exp-thm}[thm]{Expected Theorem}
\newtheorem{exp-lem}[thm]{Expected Lemma}
\newtheorem{exp-cor}[thm]{Expected Corollary}
\newtheorem{lem}[thm]{Lemma}
\newtheorem{prop}[thm]{Proposition}    
\newtheorem{cor}[thm]{Corollary}
\newtheorem{defin}[thm]{Definition}
\theoremstyle{remark}                 
\newtheorem{remark}[thm]{Remark}

\newcommand{\R}{\mathbb{R}}

\newcommand{\N}{\mathbb{N}}

\renewcommand{\epsilon}{\varepsilon}

\newcommand{\e}{\varepsilon}

\renewcommand{\leq}{\leqslant}
\renewcommand{\le}{\leqslant}
\renewcommand{\geq}{\geqslant}
\renewcommand{\ge}{\geqslant}
\allowdisplaybreaks

\makeatletter

\DeclareMathOperator*{\diverg}{div}

\makeatletter
\def\cleardoublepage{\clearpage\if@twoside \ifodd\c@page\else
	\hbox{}
	\thispagestyle{empty}
	\newpage
	\if@twocolumn\hbox{}\newpage\fi\fi\fi}
\makeatother

\numberwithin{equation}{section}

\begin{document}
	
	\title{Lipschitz regularity of almost minimizers in one-phase problems driven by the $p$-Laplace operator}
	\author{Serena Dipierro}
	\address{Serena Dipierro:  Department of Mathematics and Statistics, University of Western Australia, 35 Stirling Hwy, Crawley WA 6009	
		Australia}
	\email{serena.dipierro@uwa.edu.au }
	\author{Fausto Ferrari}
	\address{Fausto Ferrari: Dipartimento di Matematica\\ Universit\`a di Bologna\\ Piazza di Porta S.Donato 5\\ 40126, Bologna-Italy}
	\email{fausto.ferrari@unibo.it }
	\author{Nicol\`o Forcillo}
	\address{Nicol\`o Forcillo: Dipartimento di Matematica\\ Universit\`a di Bologna\\ Piazza di Porta S.Donato 5\\ 40126, Bologna-Italy}
	\email{nicolo.forcillo2@unibo.it }
	\author{Enrico Valdinoci}
	\address{Enrico Valdinoci: Department of Mathematics and Statistics, University of Western Australia, 35 Stirling Hwy, Crawley WA 6009	
		Australia }
	\email{enrico.valdinoci@uwa.edu.au  }
	\thanks{S.D. and E.V. are members of AustMS. F.F. and N.F. are members of INdAM.
		S.D.  has been supported by
		the Australian Research Council DECRA DE180100957
		``PDEs, free boundaries and applications''.
F.F. and N.F. have been supported by 2022- INDAM- GNAMPA project {\it Regolarit\`a locale e globale per problemi completamente nonlineari.}
N.F. has been supported by GHAIA Horizon 2020 MCSA
RISE programme grant No 777822.	
E.V. has been supported by
		the Australian Laureate Fellowship
		FL190100081 ``Minimal surfaces, free boundaries and partial differential equations''.
		Part of this work has been completed during a very pleasant visit of N.F.
		to the University of Western Australia, that we thank for the warm hospitality.
The authors wish to thank Daniela De Silva and Ovidiu Savin for useful discussions.
	}
	\date{\today}
	
	\maketitle
	
	\begin{abstract}We prove that, given~$p>\max\left\{\frac{2n}{n+2},1\right\}$, the nonnegative
	almost minimizers of the nonlinear free boundary functional
	$$		J_p(u,\Omega):=\int_{\Omega}\Big( |\nabla u(x)|^p+\chi_{\{u>0\}}(x)\Big)\,dx$$
are Lipschitz continuous.
\end{abstract}
	
	\section{Introduction}

In this article we consider a nonlinear free boundary problem and we establish the Lipschitz continuity
of its almost minimizers. The classical motivations for free boundary problems of these
types stem from  flows with jets and cavities (see e.g. Section~1.1 in~\cite{MR2145284}).
In this context, the conditions arising from a free boundary problem can be seen
as the variational counterpart
of Bernoulli's law according to which pressure is prescribed on the free streamline
as a balance with the velocity (or the kinetic energy) of the fluid.

The nonlinear feature of the corresponding differential operator
aims at modeling possibly non-Newtonian fluids, in which linear relations between physical quantities
are replaced by power-laws.

Similar models also appear in the study of electrical impedance tomography (see~\cite{MR1607608}),
optimal heat flows (see~\cite{MR473960}), electrochemical machining (see~\cite{MR983738}) and
high activation energy in combustion theory (see~\cite{MR1044809, MR1976085, MR2237691}).
Also, free boundary problems can be used as a sharp-interface approximation
of phase coexistence models (see e.g.~\cite{MR2126143, MR2139200, MR2262703}, roughly speaking
as a replacing for phases modeled by a smooth state parameter with a merely continuous one).\medskip

The setting of almost minimizers is also classical in the calculus of variations
(dating back, at least up to a certain extent, to a famous sentence in Leibniz's {\em Specimen Geometriae Luciferae},
probably written in the mid-1690s, ``pro minimis adhiberi possunt quasi minima'', that is
``the almost minimizers can be exploited in place of minimizers'').\medskip

More specifically, the mathematical setting that we consider here goes as follows.
Let 
	 $\Omega\subset\R^n$ be a given domain and $p>\max\left\{\frac{2n}{n+2},1\right\}$.
	We consider the energy functional
	\begin{equation}\label{defjp}
		J_p(u,\Omega):=\int_{\Omega}\Big( |\nabla u(x)|^p+\chi_{\{u>0\}}(x)\Big)\,dx
	\end{equation}
	for all $u\in W^{1,p}(\Omega)$ with $u\ge0$.
	
The condition that~$u$ is nonnegative corresponds, in the framework of free boundary problems,
to considering ``one-phase'' solutions (solutions which may change sign being related to ``two-phase'' problems).

\medskip

The precise notion of almost minimizers that we use in this paper is the following one:

\begin{defin}\label{DE:QA} Let~$\kappa\ge0$ and~$\beta>0$. We say that~$u\in W^{1,p}(\Omega)$ is an almost minimizer for~$J_p$ in~$\Omega$,
		with constant~$\kappa$ and exponent~$\beta$, if~$u\geq 0$ a.e. in~$\Omega$ and
		\begin{equation}\label{inequality-of-J-p-u-alm-min}
			J_p(u,B_\varrho(x))\leq(1+\kappa \varrho^\beta)J_p(v,B_\varrho(x)),
		\end{equation}
		for every ball~$B_\varrho(x)$ such that~$\overline{B_\varrho(x)}\subset \Omega$ and for every~$v\in W^{1,p}(B_\varrho(x))$ such that~$v=u$ on~$\partial B_\varrho(x)$ in the sense of the trace.
	\end{defin}
		
In some sense, Definition~\ref{DE:QA} is one of the possible modern formalizations
of Leibniz's initial intuition reported at the beginning of this paper: namely,
almost minimizers are natural objects to look at, for instance, to deal with minimizers
of ``perturbed'' functionals. As a concrete example, if we consider
$$ \widetilde J_p(u,\Omega):=J_p(u,\Omega)+\iint_{\Omega \times\Omega}\Phi(u(y))\,\Phi(u(z))\,\Phi(u(y)-u(z))\,dy\,dz$$
for a function~$\Phi:\R\to[0,1]$ with~$\Phi=0$ in~$(-\infty,0]$, we readily see that~$ J_p(u,B_r(x))\le\widetilde J_p(u,B_r(x))$
and
\begin{eqnarray*}&&
\widetilde J_p(u,B_r(x))\le J_p(u,B_r(x))+
\iint_{B_r(x) \times B_r(x)}\chi_{\{u>0\}}(y)\,\chi_{\{u>0\}}(z)\,dy\,dz\\&&\qquad\le
J_p(u,B_r(x))+|B_r|\int_{B_r(x) }\chi_{\{u>0\}}(y)\,dy
\le (1+|B_1|\,r^n)J_p(u,B_r(x)).
\end{eqnarray*}
Accordingly, a minimizer for the ``complicated'' functional~$\widetilde J_p$ turns out
to be an almost minimizer for the ``simpler'' functional~$J_p$.\medskip

As usual,
the constants depending only on~$n$ and~$p$ are called universal. If~$u$ is an almost minimizer, the structural constants may depend on~$\kappa$ and~$\beta$ as well.\medskip

Our main result establishes the Lipschitz regularity of the one-phase almost minimizers as follows:
	
		\begin{thm}\label{theor-Lipsch-contin-alm-minim}
		Let~$p>\max\left\{\frac{2n}{n+2},1\right\}$ and~$u$ be an almost minimizer
		for~$J_p$ in~$B_1$ with constant~$\kappa$ and exponent~$\beta$. 

Then,
		\[\left\|\nabla u\right\|_{L^{\infty}(B_{1/2})}\le C\Big(\left\|\nabla u\right\|_{L^p(B_1)}+1\Big),\]
		where~$C>0$ is a constant depending on~$n$, $p$,
		$\kappa$ and~$\beta$.
		
		In addition, $u$ is uniformly Lipschitz continuous in a neighborhood of~$\left\{u=0\right\}$,
		namely if~$u(0)=0$ then
		\[\left|\nabla u\right|\le C\quad \mbox{in }B_{r_0},\]
		for some~$C>0$, depending only on~$n$, $p$, $\kappa$ and~$\beta$, and~$r_0\in(0,1)$, depending
		on~$n$, $p$, $\kappa$, $\beta$ and~$\|\nabla u\|_{L^p(B_1)}$.
\end{thm}

We stress that Theorem~\ref{theor-Lipsch-contin-alm-minim} is new for~$p\ne2$, the case~$p=2$ being treated in~\cite{DS}.
\medskip

We recall that, when~$p=2$, minimizers of~\eqref{defjp} were studied in~\cite{MR618549}, where the
	Lipschitz regularity
	of minimizers and the regularity of flat free boundaries were established.
	In~\cite{MR990856, MR973745, MR1029856}
	Caffarelli developed a viscosity approach to the free boundary problem for~$p=2$ in the two-phase
	setting (see also~\cite{D} and the references therein for related
	free boundary regularity properties). For one-phase problems the viscosity approach for operators governed by the~$p$-Laplace or~$p(x)-$Laplace operators
	has been developed in~\cite{GleRic, FL}
	and in~\cite{FL2} as well, where Lipschitz regularity of viscosity solutions of inhomogeneous one-phase problems, governed by the~$p(x)-$Laplace operator and some regularity properties of their free boundaries were established.\medskip
		
	The Lipschitz regularity of the minimizers of the functional in~\eqref{defjp} has been obtained
	in~\cite{MR3771123}, which has also provided a proof of the Lipschitz
	regularity when~$p=2$ without using
	monotonicity formulae.
We refer to~\cite{MR3910196} as well, where a discrete version of the Weiss monotonicity formula
	for the functional in~\eqref{defjp}
	has been established for~$p$ close to~$2$.  \medskip
	
	Minimizers of the functional in~\eqref{defjp} have been also considered in~\cite{MR2133664},
	in which the regularity of the free boundary near flat points was established.
	See also~\cite{MR2250499, MR2383537, MR2431665, MR2680176, MR2876248, MR3168631, MR3249814, MR3390082, MR4266232} for regularity results on~$p$-Laplacian free boundary problems.
	\medskip
	
Regarding the setting of almost minimizers, the case~$p=2$ has been investigated
	in~\cite{MR3385167, MR3948692, DS}.
	The case of almost minimizers for~$p\ne2$ was, to the best of our knowledge, not fully investigated, hence Theorem~\ref{theor-Lipsch-contin-alm-minim} aims at starting a research line in this direction,
and the results presented in this article are part of the PhD thesis of the third author~\cite{F}.
In this paper we take inspiration from the approach used in~\cite{DS},
see also~\cite{MR4062979, MR4201786}.
\medskip

We point out that the case~$p\ne2$ provides significant technical complications,
especially due to the fact that the sum of two solutions is not a solution any longer,
thus making it difficult to develop an exhaustive theory of harmonic replacements in the nonlinear scenario.
\medskip

We highlight the fact that the regularity results put forth in this paper do not follow from the classical ones in the calculus of variations, since the integrand of the functional that we consider here is a discontinuous function, due to the presence of the term~$\chi_{\{u>0\}}$ (instead, the classical cases dealt with require the integrand to be continuous, or even Lipschitz continuous, see e.g. assumption~(8.48) in~\cite{G}).

\bigskip

	The paper is organized as follows. In Section~\ref{KMD-2rkXmfTO} we recall the notion of
$p$-harmonic replacement and put forth some basic energy estimates needed in our main arguments.

Then, in Section~\ref{DSIKD:ICH}, we develop a dichotomy theory according to which, roughly speaking,
the average of the energy of an almost minimizer decreases in a smaller ball, unless we are arbitrarily close
to the case of linear functions.

In Section~\ref{sec:lip1} we show that
	almost minimizers of~$J_p$ are Lipschitz continuous, namely we prove Theorem~\ref{theor-Lipsch-contin-alm-minim}.
	
\section{The $p$-harmonic replacement}\label{KMD-2rkXmfTO}

One of the classical ingredients in nonlinear partial differential
equations is the notion of~$p$-harmonic replacement, which we now recall:
	
	\begin{defin}
		Let~$r>0$, $x_0\in\R^n$ and~$u\in W^{1,p}(B_r(x_0))$. We say that~$v\in W^{1,p}(B_r(x_0))$ is the~$p$-harmonic replacement of~$u$ in~$B_r(x_0)$ if
		\[\int_{B_r(x_0)}\left|\nabla v\right|^p \,dx=\min_{u-w\in W^{1,p}_0(B_r(x_0))}\int_{B_r(x_0)}\left|\nabla w\right|^p \,dx.\]
	\end{defin}
	
	We remark that if~$v$ is the~$p$-harmonic replacement of~$u$ in~$B_r(x_0)$, then in particular it satisfies
	\begin{equation}\label{weak-definition-of-p-harmonic-function}
		\int_{B_r(x_0)}\left|\nabla v(x)\right|^{p-2}\nabla v(x)\cdot \nabla\varphi(x)\,dx=0
	\end{equation}
	for all~$\varphi\in W^{1,p}_{0}(B_r(x_0))$, namely~$v$ is a weak solution of~$\Delta_p v=0$
	in~$B_r(x_0)$.
	
	We now provide some energy estimates of classical flavor for the~$p$-harmonic replacement
	that we will use in the sequel.
	
	\begin{lem}\label{lemma-p-harm-replac}
		Let~$x_0\in\mathbb{R}^{n}$, $r>0$ and~$u\in W^{1,p}(B_r(x_0))$.
		Let~$v$ be the~$p$-harmonic replacement of~$u$ in~$B_r(x_0)$.
		Then,
		\begin{itemize}
			\item[(i)] if~$1<p<2,$ then
			\begin{equation}\label{first-ineq-lemma-p-harm-repl}\begin{split} &
					\int_{B_r(x_0)}\left|\nabla u(x)-\nabla v(x)\right|^p\,dx\\
					&\qquad\le C\left(\int_{B_r(x_0)}\left(\left|\nabla u(x)\right|^p-\left|\nabla v(x)\right|^p\right)\,dx\right)^{\frac{p}{2}}
					\left(\int_{B_r(x_0)}\big(\left|\nabla u(x)\right|+\left|\nabla v(x)\right|\big)^p\,dx\right)^{1-\frac{p}{2}},\end{split}
			\end{equation}
			for some positive universal constant~$C$;
			\item[(ii)] if~$p\ge 2,$ then
			\begin{equation}\label{second-ineq-lemma-p-harm-repl}
				\int_{B_r(x_0)}\left|\nabla u(x)-\nabla v(x)\right|^p\,dx
				\le C\int_{B_r(x_0)}\big(\left|\nabla u(x)\right|^p-\left|\nabla v(x)\right|^p\big)\,dx,
			\end{equation}
			for some positive universal constant~$C$.
		\end{itemize}		
	\end{lem}	
	
	\begin{proof}
		For all~$s\in[0,1]$, we consider the family of functions~$
		u_s(x):=su(x)+(1-s)v(x)$.
		Notice that~$u_0=v$ and~$u_1=u$. As a consequence,
		\begin{align*}
			\int_{B_r(x_0)}\Big(\left|\nabla u(x)\right|^p-\left|\nabla v(x)\right|^p\Big)\,dx
			&=\int_{B_r(x_0)}\left(\int_0^1\frac{d}{ds}\left|\nabla u_s(x)\right|^p\,ds\right)\,dx\\
			&=\int_{B_r(x_0)}\left(\int_0^1 p\left|\nabla u_s(x)\right|^{p-2}\nabla u_s(x)\cdot \nabla(u-v)(x)\,ds\right)\,dx.
		\end{align*}
		This and~\eqref{weak-definition-of-p-harmonic-function} (used here with~$\varphi:=u-v$) lead to
		\begin{align*}
			&\int_{B_r(x_0)}\Big(\left|\nabla u(x)\right|^p-\left|\nabla v(x)\right|^p\Big)\,dx
			\\=\;&p\,\left[\int_{B_r(x_0)}\left(\int_0^1\left|\nabla u_s(x)\right|^{p-2}
			\nabla u_s(x)\cdot \nabla(u-v)(x)\,ds\right)\,dx\right.
			\\&\qquad\qquad\qquad \left.
			-\int_0^1\left(\int_{B_r(x_0)}\left|\nabla v(x)\right|^{p-2}\nabla v(x)\cdot \nabla(u-v)(x)\,dx\right)\,ds\right]\\
			=\;&p\int_0^1\bigg(\int_{B_r(x_0)}\left(\left|\nabla u_s(x)\right|^{p-2}\nabla u_s(x)
			-\left|\nabla v(x)\right|^{p-2}\nabla v(x)\right)\cdot\nabla(u-v)(x)\,dx\bigg)\,ds.
		\end{align*}
		We also point out that
		\begin{equation}\label{difference-u-^-s-v}
			u_s(x)-v(x)=s(u(x)-v(x)),
		\end{equation}
		and therefore we arrive at
		\begin{equation}\begin{split}\label{lower-bound-integral-nabla-u-nabla-v-1}
				&\int_{B_r(x_0)}\Big(\left|\nabla u(x)\right|^p-\left|\nabla v(x)\right|^p\Big)\,dx\\
				=\;&p\int_0^1\frac{1}{s}\bigg(\int_{B_r(x_0)}\left(\left|\nabla u_s(x)\right|^{p-2}\nabla u_s(x)
				-\left|\nabla v(x)\right|^{p-2}\nabla v(x)\right)\cdot\nabla(u_s-v)(x)\,dx\bigg)\,ds.
		\end{split}\end{equation}
		
		Now we recall the inequality 
		\[\big(\left|\xi\right|^{p-2}\xi-\left|\zeta\right|^{p-2}\zeta\big)\cdot(\xi-\zeta)\ge \gamma
		\begin{cases}
			\left|\xi-\zeta\right|^2\left(\left|\xi\right|+\left|\zeta\right|\right)^{p-2}&\mbox{if }1<p< 2,\\
			\left|\xi-\zeta\right|^p&\mbox{if }p\ge2,
		\end{cases}\]
		for any~$\xi$, $\zeta\in \mathbb{R}^n\setminus\{0\}$, for some
		positive universal constant~$\gamma$, see e.g. page~100 in~\cite{MR2133664}.
		Thus, using this inequality with~$\xi:=\nabla u_s$ and~$\zeta:=\nabla v$
		into~\eqref{lower-bound-integral-nabla-u-nabla-v-1}, we obtain that
		\begin{equation*}\begin{split}
				&\int_{B_r(x_0)}\Big(\left|\nabla u(x)\right|^p-\left|\nabla v(x)\right|^p\Big)\,dx\\
				\ge\;& 
				\begin{cases}
					\displaystyle p\,\gamma\int_0^1\frac{1}{s}\left(\int_{B_r(x_0)}\left|\nabla u_s(x)
					-\nabla v(x)\right|^2(\left|\nabla u_s(x)\right|
					+\left|\nabla v(x)\right|)^{p-2}\,dx\right)\,ds,&\; \mbox{if }1<p< 2,
					\\
					\displaystyle p\,\gamma\int_0^1\frac{1}{s}\left(\int_{B_r(x_0)}\left|\nabla u_s(x)-\nabla v(x)\right|^p \,dx\right)\,ds,&\;
					\mbox{if }p\ge 2.
				\end{cases}
		\end{split}\end{equation*}
		Furthermore, recalling~\eqref{difference-u-^-s-v}, we conclude that
		\begin{equation}\begin{split}\label{lower-bound-integral-nabla-u-nabla-v-2}
				&\int_{B_r(x_0)}\Big(\left|\nabla u(x)\right|^p-\left|\nabla v(x)\right|^p\Big)\,dx\\
				\ge\;& 
				\begin{cases}
					\displaystyle p\,\gamma\int_0^1 {s}\left(\int_{B_r(x_0)}\left|\nabla u(x)
					-\nabla v(x)\right|^2(\left|\nabla u_s(x)\right|
					+\left|\nabla v(x)\right|)^{p-2}\,dx\right)\,ds,&\; \mbox{if }1<p< 2,
					\\
					\displaystyle p\,\gamma\int_0^1 {s}^{p-1}\left(\int_{B_r(x_0)}\left|\nabla u(x)-\nabla v(x)\right|^p \,dx\right)\,ds,&\;
					\mbox{if }p\ge 2.
				\end{cases}
		\end{split}\end{equation}
		
		Now, if~$1<p<2$, since~$s\in[0,1]$ we have that
		$$\left|\nabla u_s\right|+\left|\nabla v\right|\le s\left|\nabla u\right|+(1-s)\left|\nabla v\right|+\left|\nabla v\right|=s\left|\nabla u\right|+(2-s)\left|\nabla v\right| {\le} 2\,\big(\left|\nabla u\right|+\left|\nabla v\right|\big).
		$$
		Hence, plugging this information into~\eqref{lower-bound-integral-nabla-u-nabla-v-2}, we find that		
		\begin{equation}\label{lower-bound-integral-nabla-u-nabla-v-1-<-p-leq-2-1}\begin{split}
				&\int_{B_r(x_0)}\Big(\left|\nabla u(x)\right|^p-\left|\nabla v(x)\right|^p\Big)\,dx\\
				\ge\;& p\,\gamma\, 2^{p-2}\int_0^1 {s}\left(\int_{B_r(x_0)} \left|\nabla u(x)-\nabla v(x)\right|^2
				\big(\left|\nabla u(x)\right|+\left|\nabla v(x)\right|\big)^{p-2}\,dx\right)\,ds\\
				=\;&p\,\gamma\, 2^{p-3} \int_{B_r(x_0)}\left|\nabla u(x)-\nabla v(x)\right|^2\big(
				\left|\nabla u(x)\right|+\left|\nabla v(x)\right|\big)^{p-2}\,dx.
		\end{split}\end{equation}
		
		We now apply the H\"older's inequality with H\"older exponent~$2/p$ and conjugate exponent
		\[\left(\frac{2}{p}\right)'=\frac{2/p}{2/p-1}=\frac{2}{2-p},\]
		to see that
		\begin{align*}
			&\int_{B_r(x_0)}\left|\nabla u(x)-\nabla v(x)\right|^p\,dx\\
			=\;&\int_{B_r(x_0)}\left|\nabla u(x)-\nabla v(x)\right|^p\big(
			\left|\nabla u(x)\right|+\left|\nabla v(x)\right|\big)^{\frac{p(p-2)}{2}}\big(
			\left|\nabla u(x)\right|+\left|\nabla v(x)\right|\big)^{-\frac{p(p-2)}{2}}\,dx\\
			\le\;&\left(\int_{B_r(x_0)}\left|\nabla u(x)-\nabla v(x)\right|^{2}\big(
			\left|\nabla u(x)\right|+\left|\nabla v(x)\right|\big)^{p-2}\,dx\right)^{\frac{p}{2}}\\&\qquad\qquad\times
			\left(\int_{B_r(x_0)}\big(\left|\nabla u(x)\right|+\left|\nabla v(x)\right|\big)^p\,dx\right)^{1-\frac{p}{2}}.
		\end{align*}
		This and~\eqref{lower-bound-integral-nabla-u-nabla-v-1-<-p-leq-2-1} yield that
		\begin{align*}
			&\int_{B_r(x_0)}\left|\nabla u(x)-\nabla v(x)\right|^p\,dx\\
			\le \;& C\left(\int_{B_r(x_0)}\big(\left|\nabla u(x)\right|^p-\left|\nabla v(x)\right|^p\big)\, 
			dx\right)^{\frac{p}{2}}\left(\int_{B_r(x_0)}\big(\left|\nabla u(x)\right|+\left|\nabla v(x)\right|\big)^p\,dx\right)^{1-\frac{p}{2}},
		\end{align*}
		for some universal constant~$C>0$, which proves the desired result in~\eqref{first-ineq-lemma-p-harm-repl} when~$1<p<2$.
		
		If instead~$p\ge 2$,
		we deduce from~\eqref{lower-bound-integral-nabla-u-nabla-v-2} that
		\begin{align*}
			\int_{B_r(x_0)}\big(\left|\nabla u(x)\right|^p-\left|\nabla v(x)\right|^p\big)\,dx
			&\ge p\,\gamma\int_0^1 {s}^{p-1}\left(\int_{B_r(x_0)}
			\left|\nabla u(x)-\nabla v(x)\right|^p \,dx\right)\,ds\\
			&=\gamma \int_{B_r(x_0)}\left|\nabla u(x)-\nabla v(x)\right|^p\,dx,
		\end{align*}
		which establishes~\eqref{second-ineq-lemma-p-harm-repl} when~$p\ge2$ and completes
		the proof of Lemma~\ref{lemma-p-harm-replac}.
	\end{proof}

\section{Some	dichotomy results}\label{DSIKD:ICH}

With the preliminary work carried out so far,
we can now provide a dichotomy statement. Roughly speaking,
either the average of the energy of an almost minimizer decreases in a smaller ball,
or the distance of its gradient and a suitable constant vector becomes as small as we wish
(that is, linear functions are the ``only ones for which the average does not improve in small balls'').
The precise result goes as follows:

	\begin{prop}\label{proposition-dichotomy}
		There exists~$\e_0\in(0,1)$ such that for every~$\varepsilon\in(0,\e_0)$
		there exist~$\eta\in(0,1)$,  $M\ge1$ and~$\sigma_0\in(0,1)$, depending on~$\e$, $n$ and~$p$, such that
		if~$\sigma\in[0, \sigma_0]$ and~$a\ge M$ then the following statement holds true.
		
		Let~$u\in W^{1,p}(B_1)$ 
		be such that
		\begin{equation}\label{inequality-of-J-p-u-sigma}
			J_p(u,B_1)\le(1+\sigma)J_p(v,B_1)
		\end{equation}
		for all~$v\in W^{1,p}(B_1)$ such that~$v=u$ on~$\partial B_1$, with
		\begin{equation}\label{average-integral-nabla-u-to-p-B-1}
			a:=\left(\fint_{B_1}\left|\nabla u(x)\right|^p \,dx\right)^{1/p}.
		\end{equation}
		Then, either
		\begin{equation}\label{first-conclusion-dichotomy}
			\left(\fint_{B_\eta}\left|\nabla u(x)\right|^p \,dx\right)^{1/p}\le \frac{a}{2},
		\end{equation}
		or
		\begin{equation}\label{second-conclusion-dichotomy} 
			\left(\fint_{B_\eta}\left|\nabla u(x)-q\right|^p \,dx\right)^{1/p}\le \varepsilon a,
		\end{equation}
		with~$q\in \mathbb{R}^n$ such that 
		\begin{equation}\label{bound-q}
			\frac{a}{4}<\left|q\right|\le C_0\,a,
		\end{equation}
		and~$C_0>0$ universal.
	\end{prop}
	
	\begin{proof}
		Let~$v$ be the~$p$-harmonic replacement of~$u$ in~$B_1$. 
		By Theorem~3.19 in~\cite{MZ} we have that, for every~$x\in B_{1/2}$,
		\begin{align*}
			\left|\nabla v(x)\right|^p&\le \sup_{B_{1/4}(x)}\left|\nabla v\right|^p\le C\fint_{B_{1/2}(x)}\left|\nabla v(y)\right|^p\,dy\le C\fint_{B_1}\left|\nabla u(y)\right|^p\,dy.
		\end{align*}
		As a consequence, for all~$x\in B_{1/2}$,
		\begin{equation}\label{bound-norm-nabla-v}
			\left|\nabla v(x)\right|\le C_0\,a,
		\end{equation}
		for some positive universal constant~$C_0$.
		
		Accordingly,	we denote by~$q:=\nabla v(0)$ and we deduce from~\eqref{bound-norm-nabla-v}
		that~$|q|\le C_0\,a$. This, together with Theorem~2 in~\cite{Man}, gives that, for all~$\eta\in(0,1/2]$,
		\begin{equation}\label{average-integral-nabla-v-q}\begin{split}
				\fint_{B_\eta}\left|\nabla v(x)-q\right|^p\,dx\le\;&
				\fint_{B_\eta}\bigg(C\bigg(\frac{\eta}{1/2}\bigg)^{\alpha}\left\|\nabla v\right\|_{L^{\infty}(B_{1/2})}\bigg)^p \,dx\\
				\le\;& C\,\eta^{\alpha p} \left\|\nabla v\right\|_{L^{\infty}(B_{1/2})}^p
				\\ \le\;& C_1\,\eta^{\alpha p}\,a^p,
		\end{split}\end{equation}
		for some~$\alpha\in(0,1)$ and~$C_1>0$ universal.
		
		Furthermore, using~\eqref{inequality-of-J-p-u-sigma}
		we have that
		\begin{equation}\label{scbevturyuitryutyu6y75868430}\begin{split}
				\int_{B_1}\big(\left|\nabla u(x)\right|^p-\left|\nabla v(x)\right|^p\big)\,dx		
				&\le J_p(u,B_1)-\int_{B_1}\left|\nabla v(x)\right|^p\,dx \\
				&\le  (1+\sigma)J_p(v,B_1)-\int_{B_1}\left|\nabla v(x)\right|^p\,dx \\
				&\le C\left(\sigma\int_{B_1}\left|\nabla v(x)\right|^p\,dx+1\right)\\
				&\le C\left(\sigma\int_{B_1}\left|\nabla u(x)\right|^p\,dx+1\right)
		\end{split}\end{equation}
		We distinguish two cases, $p\ge2$ and~$1<p<2$. 
		
		If~$p\ge 2$, then from~\eqref{second-ineq-lemma-p-harm-repl} and~\eqref{scbevturyuitryutyu6y75868430}
		we deduce that
		\begin{align*}
			\int_{B_1}\left|\nabla u(x)-\nabla v(x)\right|^p\,dx			
			&\le C\left(\sigma\int_{B_1}\left|\nabla u(x)\right|^p\,dx+1\right).
		\end{align*}
		Taking the average over~$B_1$, we thereby obtain that
		\begin{equation*}
			\fint_{B_1}\left|\nabla u(x)-\nabla v(x)\right|^p\,dx\le C(\sigma a^p+1).
		\end{equation*}
		Using this and~\eqref{average-integral-nabla-v-q}, we get that
		\begin{equation}\label{average-integral-nabla-u-q-p->-2}\begin{split}
				\fint_{B_\eta}\left|\nabla u(x)-q\right|^p\,dx &\le2^{p-1}
				\fint_{B_\eta}\big(\left|\nabla u(x)-\nabla v(x)\right|^p+\left|\nabla v(x)-q\right|^p\big)\,dx\\
				&\le 2^{p-1}\left(\frac{\left|B_1\right|}{\left|B_\eta\right|}\fint_{B_1}\left|\nabla u(x)-\nabla v(x)\right|^p\,dx+C_1\,\eta^{\alpha p}\,a^p\right)
				\\&\le 2^{p-1} \Big(C\eta^{-n}(\sigma a^p+1)+C_1a^p\eta^{\alpha p}\Big),\\
				&= 2^{p-1}C \eta^{-n}\sigma a^p+2^{p-1}C\eta^{-n}+2^{p-1}C_1a^p\eta^{\alpha p}.
			\end{split}
		\end{equation}
		This yields that
		\begin{equation}\label{average-integral-nabla-u-to-p-B-eta-p->-2}
			\fint_{B_\eta}\left|\nabla u(x)\right|^p\,dx\le 2^{2(p-1)}C\eta^{-n}\sigma a^p+2^{2(p-1)}C\eta^{-n}+2^{2(p-1)}C_1a^p\eta^{\alpha p}+2^{p-1}\left|q\right|^p.
		\end{equation}
		
		If instead~$1<p< 2$, by virtue of~\eqref{first-ineq-lemma-p-harm-repl}
		and~\eqref{scbevturyuitryutyu6y75868430}, we obtain that
		\begin{align*}
			&\int_{B_1}\left|\nabla u(x)-\nabla v(x)\right|^p\,dx\\
			\le\;& C\left(\sigma
			\int_{B_1}\left|\nabla u(x)\right|^p\,dx+1\right)^{\frac{p}{2}}
			\left(\int_{B_1}\big(\left|\nabla u(x)\right|+\left|\nabla v(x)\right|\big)^p\,dx\right)^{1-\frac{p}{2}}\\
			\le\;& C \left(\sigma^{\frac{p}{2}}\left(\int_{B_1}\left|\nabla u(x)\right|^p\,dx\right)^{\frac{p}{2}}+1\right)\left(\int_{B_1}\big(\left|\nabla u(x)\right|^p+\left|\nabla v(x)\right|^p\big)\,dx\right)^{1-\frac{p}{2}}.
		\end{align*}
		Thus, since~$v$ is the~$p$-harmonic replacement of~$u$ in~$B_1$,
		\begin{align*}
			&\int_{B_1}\left|\nabla u(x)-\nabla v(x)\right|^p\,dx\\
			\le\;& C\left(\sigma^{\frac{p}{2}}\left(\int_{B_1}\left|\nabla u(x)\right|^p\,dx\right)^{\frac{p}{2}}+1\right)
			\left(\int_{B_1}\left|\nabla u(x)\right|^p\,dx\right)^{1-\frac{p}{2}}\\
			=\;&C \left(\sigma^{\frac{p}{2}}\int_{B_1}\left|\nabla u(x)\right|^p\,dx+\left(\int_{B_1}\left|\nabla u(x)\right|^p\,dx\right)^{1-\frac{p}{2}}\right).
		\end{align*}
		Consequently, taking the average integral, we have that
		\begin{equation*}
			\fint_{B_1}\left|\nabla u(x)-\nabla v(x)\right|^p\,dx\le C\left(
			\sigma^{\frac{p}{2}}a^p+|B_1|^{-\frac{p}{2}}a^{p\left(1-\frac{p}{2}\right)}\right)
			\le C\left(\sigma^{\frac{p}{2}}a^p+a^{p\left(1-\frac{p}{2}\right)}\right).
		\end{equation*}
		{F}rom this and~\eqref{average-integral-nabla-v-q}, we obtain that
		\begin{equation}\label{average-integral-nabla-u-q-1-<p-leq-2}\begin{split}
				\fint_{B_\eta}\left|\nabla u(x)-q\right|^p\,dx\le\;& 2^{p-1}
				\fint_{B_\eta}\big(\left|\nabla u(x)-\nabla v(x)\right|^p+\left|\nabla v(x)-q\right|^p\big)\,dx\\
				\le\;& 2^{p-1}\left(\frac{\left|B_1\right|}{\left|B_\eta\right|}\fint_{B_1}\left|\nabla u(x)-\nabla v(x)\right|^p\,dx+C_1\,\eta^{\alpha p}\,a^p\right)\\
				\le \;&
				2^{p-1}C\eta^{-n}\sigma^{\frac{p}{2}}a^p+2^{p-1}C\eta^{-n}a^{p\left(1-\frac{p}{2}\right)}
				+2^{p-1}C_1a^p\eta^{\alpha p}.
		\end{split}\end{equation}
		This gives that
		\begin{equation}\label{average-integral-nabla-u-to-p-B-eta-1<-p-leq-2}
			\fint_{B_\eta}\left|\nabla u(x)\right|^p\,dx\le 2^{2(p-1)}C\eta^{-n}\sigma^{\frac{p}{2}} a^p+2^{2(p-1)}C\eta^{-n}
			a^{p\left(1-\frac{p}{2}\right)}+2^{2(p-1)}C_1a^p\eta^{\alpha p}+2^{p-1}\left|q\right|^p.
		\end{equation}
		
		Now, 
		given~$\varepsilon_0\in(0, 1/4]$, we claim that 
		for every~$\varepsilon\in (0,\epsilon_0)$ there exists~$\eta$ small enough (depending on~$\varepsilon$)
		such that if~$\sigma$ is chosen sufficiently small and~$a$ sufficiently large (depending on~$\eta$,
		and thus on~$\e$) then
		\begin{equation}\label{conditions-on-eta-sigma-a}
			\begin{cases}
				2^{2(p-1)}C\eta^{-n}\sigma a^p+2^{2(p-1)}C\eta^{-n}+2^{2(p-1)}C_1a^p\eta^{\alpha p}\le 2^{p-1}\varepsilon^pa^p\le \frac{a^p}{2^{p+1}}\\
				\mbox{if }p\ge 2,\\
				\\
				2^{2(p-1)}C\eta^{-n}\sigma^{\frac{p}{2}} a^p+2^{2(p-1)}C\eta^{-n}
				a^{p\left(1-\frac{p}{2}\right)}+2^{2(p-1)}C_1a^p\eta^{\alpha p}\le 2^{p-1}\varepsilon^pa^p\le \frac{a^p}{2^{p+1}}\\
				\mbox{if }1<p< 2.
			\end{cases}
		\end{equation}
To prove this we distinguish two cases. If~$p\ge 2,$ we pick~$\eta>0$ sufficiently small
such that~$\varepsilon^p-2^{p-1}C\eta> 2^{p-1}C_1\eta^{\alpha p}$. This allows us to define
		$$
		M:=\left(\frac{2^{p-1}C\eta^{-n}}{\varepsilon^p-2^{p-1}C\eta-2^{p-1}C_1\eta^{\alpha p}}\right)^{1/p}
		.$$
		Note also that we can suppose~$M\ge1$ by taking~$\eta$ small enough.
		Let also 
		$$ \sigma_0:=\eta^{n+1}.$$ With this setting, we obtain that, for every~$a\geq M$ and for every~$0<\sigma\leq \sigma_0,$
		\begin{eqnarray*}&&
2^{2(p-1)}C\eta^{-n}\sigma a^p+ 2^{2(p-1)}C\eta^{-n}+ 2^{2(p-1)}C_1a^p\eta^{\alpha p}\\&\leq
			& a^p\Big(  2^{2(p-1)}C\eta+ 2^{2(p-1)}C_1\eta^{\alpha p}\Big)
			+ 2^{2(p-1)}C\eta^{-n}\\&=&
	a^p\Big(  2^{2(p-1)}C\eta+ 2^{2(p-1)}C_1\eta^{\alpha p}\Big)
			+2^{p-1}\,M^p\Big(\varepsilon^p-2^{p-1}C\eta-2^{p-1}C_1\eta^{\alpha p}\Big)
			\\
			&\leq& a^p\Big(  2^{2(p-1)}C\eta+ 2^{2(p-1)}C_1\eta^{\alpha p}\Big)
			+2^{p-1}\,a^p\Big(\varepsilon^p-2^{p-1}C\eta-2^{p-1}C_1\eta^{\alpha p}\Big)\\&=&
			2^{p-1}\varepsilon^p a^p+ a^p
			\Big(  2^{2(p-1)}C\eta+ 2^{2(p-1)}C_1\eta^{\alpha p}
			- 2^{2(p-1)} C\eta- 2^{2(p-1)}C_1\eta^{\alpha p}\Big)\\&=&2^{p-1}\varepsilon^p a^p,
		\end{eqnarray*}
		which proves~\eqref{conditions-on-eta-sigma-a} when~$p\ge 2$.
		
		If instead~$1<p<2,$ we pick~$\eta>0$ small enough such that~$\varepsilon^p>2^{p-1}C\eta+2^{p-1}C_1\eta^{\alpha p}.$ In this way, we can define
		$$M:=
		\left(\frac{C2^{p-1}\eta^{-n}}{\varepsilon^p-2^{p-1}C\eta-2^{p-1}C_1\eta^{\alpha p}}\right)^{2/p^2}.$$
		Let also~$\sigma_0:=\eta^{(n+1)\frac{2}{p}}.$
Then, whenever~$a\geq M$ and~$0<\sigma\leq \sigma_0$ it follows that
		\begin{eqnarray*}
			&& 2^{2(p-1)}C\eta^{-n}\sigma^{\frac{p}{2}}a^p+ 2^{2(p-1)}C\eta^{-n} a^{p\left(1-\frac{p}{2}\right)}+ 2^{2(p-1)}C_1a^p\eta^{\alpha p}\\&\leq& 2^{2(p-1)}C\eta a^p+ 2^{2(p-1)}C\eta^{-n} a^{p\left(1-\frac{p}{2}\right)
			}+ 2^{2(p-1)}C_1a^p\eta^{\alpha p}\\&=&
			2^{p-1} a^p\Big( 2^{p-1}C\eta + 2^{p-1}C\eta^{-n} a^{-\frac{p^2}{2}	}+ 2^{p-1}C_1\eta^{\alpha p}\Big)\\&\leq&
2^{p-1} a^p\Big( 2^{p-1}C\eta + 2^{p-1}C\eta^{-n} M^{-\frac{p^2}{2}	}+ 2^{p-1}C_1\eta^{\alpha p}\Big)\\&=&
			2^{p-1} a^p\Big( 2^{p-1}C\eta + \varepsilon^p- 2^{p-1}C\eta- 2^{p-1}C_1\eta^{\alpha p}
			+ 2^{p-1}C_1\eta^{\alpha p}\Big)\\&=&  2^{p-1}\e^pa^p,
		\end{eqnarray*}
		which establishes~\eqref{conditions-on-eta-sigma-a} in the case~$1<p<2$ as well.
		
		In order to complete the proof of Proposition~\ref{proposition-dichotomy},
		we now distinguish two cases according to the size of~$\left|q\right|.$ More precisely, we first suppose that
		\[\left|q\right|\le\frac{a}{4}.\]
		Then, we use either~\eqref{average-integral-nabla-u-to-p-B-eta-p->-2} (if~$p\ge2$) or~\eqref{average-integral-nabla-u-to-p-B-eta-1<-p-leq-2} (if~$1<p<2$), and~\eqref{conditions-on-eta-sigma-a} to conclude that
		\begin{align*}
			\fint_{B_\eta}\left|\nabla u(x)\right|^p\,dx\le \frac{a^p}{2^{p+1}}+ 2^{p-1}\frac{a^p}{2^{2p}}=\frac{a^p}{2^{p+1}}+\frac{a^p}{2^{p+1}}=\frac{a^p}{ 2^{p}},
		\end{align*}
		and thus
		\[\left(\fint_{B_\eta}\left|\nabla u(x)\right|^p\,dx\right)^{1/p}\le\frac{a}{2},\]
		which is the first alternative in~\eqref{first-conclusion-dichotomy}.
		
		Otherwise, it holds that
		\[ \frac{a}{4}<\left|q\right|\le C_0\,a,\]
		and therefore, by either~\eqref{average-integral-nabla-u-q-p->-2} (if~$p\ge2$) or~\eqref{average-integral-nabla-u-q-1-<p-leq-2} (if~$1<p<2$),
		and~\eqref{conditions-on-eta-sigma-a}, we have that
		\[\left(\fint_{B_\eta}\left|\nabla u(x)-q\right|^p\,dx\right)^{1/p}\le \varepsilon a,\]
		which is the second alternative in~\eqref{second-conclusion-dichotomy}.
		The proof of Proposition~\ref{proposition-dichotomy} is thereby complete.
	\end{proof}	
	
	We will now show that the alternative in~\eqref{second-conclusion-dichotomy}
	can be ``improved" when~$\varepsilon$ and~$\sigma$ are sufficiently small.
	This result is the counterpart of Lemma~2.3 in~\cite{DS} in the more general setting dealt with in this paper.
	The main difficulty here with respect to Lemma~2.3 in~\cite{DS} relies on the fact that
	the problem is not linear when~$p\neq2$, and therefore,
	if~$v_1$ and~$v_2$ are the~$p$-harmonic replacements of~$u_1$
	and~$u_2$ in~$B_1$, then it is not true that~$v_1+v_2$
	is the~$p$-harmonic replacement of~$u_1+u_2$, unless~$p=2$. 
	
	We will overcome this difficulty by showing that a ``uniformly elliptic'' equation
	is still satisfied from the sum of~$p$-harmonic replacements.
	More precisely, we will show the following result:
	
	\begin{lem}\label{lemma:unifell}
		Let~$q\in\R^n$ and let~$\eta:\R^n\to\R^n$ be such that~$|\eta(x)|<\frac{|q|}2$.
		Let~$F(z):=|z|^{p-2}z$ and
		\[ A(x):=\int_0^1DF\big(q+t\eta(x) \big)\,dt. \]
		
		Then,
		$$ \lambda\,|q|^{p-2}\le A\xi \cdot \xi \le\Lambda\, |q|^{p-2},$$
		for all~$\xi\in\partial B_1$, for some~$\Lambda\ge\lambda>0$, depending on~$p$.
	\end{lem}
	
	\begin{proof}
		We notice that, for all~$t\in(0,1)$,
		\begin{equation}\begin{split}\label{sod30v9b54y8b67vb985 78by75y89}
				&|q+t\eta|\le |q|+|\eta|< |q|+\frac{|q|}2=\frac{3|q|}2\\
				{\mbox{and }}\quad& |q+t\eta|\ge |q|-|\eta|>|q|-\frac{|q|}2=\frac{|q|}2.
		\end{split}\end{equation}
		
		Furthermore, we observe that
		$$
		DF(z)=(p-2)|z|^{p-4}z\otimes z +|z|^{p-2}{\text{Id}},
		$$
		where~${\text{Id}}$ denotes the identity matrix in dimension~$n$, and therefore, for all~$\xi\in \partial B_1$,
		\begin{eqnarray*}
			DF(z)\xi \cdot\xi&=&(p-2)|z|^{p-4}(z\cdot\xi)^2+|z|^{p-2}|\xi|^2\\&\ge&
			-(2-p)^+|z|^{p-2}+|z|^{p-2}|\xi|^2\\&=&\big[1-(2-p)^+\big]\,|z|^{p-2}|\xi|^2
			\\&=&\big[1-(2-p)^+\big]\,|z|^{p-2}
			.\end{eqnarray*}
		Consequently,	
		\begin{eqnarray*}
			&& A\xi \cdot \xi\ge \big[1-(2-p)^+\big]
			\int_0^1|q+t\eta|^{p-2} \,dt.
		\end{eqnarray*}
		Accordingly, we can use~\eqref{sod30v9b54y8b67vb985 78by75y89} to find that
		$$ A\xi \cdot \xi\ge \frac{1-(2-p)^+}{2^{p-2}}  \,|q|^{p-2}
		$$
		if~$p\ge2$, and
		$$ A\xi \cdot \xi\ge \big[1-(2-p)^+\big]\left(\frac32\right)^{p-2}|q|^{p-2}
		$$
		if~$p\in(1,2)$.
		
		Similarly, since
		\begin{eqnarray*}
			DF(z)\xi \cdot\xi&\le&
			(p-2)^+|z|^{p-2}+|z|^{p-2}|\xi|^2\\&=&\big[1+(p-2)^+\big]\,|z|^{p-2}|\xi|^2
			\\&=&\big[1+(p-2)^+\big]\,|z|^{p-2}
			,\end{eqnarray*}
		we have that
		\begin{eqnarray*}
			&& A\xi \cdot \xi\le \big[1+(p-2)^+\big]
			\int_0^1|q+t\eta|^{p-2} \,dt.
		\end{eqnarray*}
		Hence, making again use of~\eqref{sod30v9b54y8b67vb985 78by75y89},
		$$ A\xi \cdot \xi\le \big[1+(p-2)^+\big] \left(\frac32\right)^{p-2}|q|^{p-2}
		$$
		if~$p\ge2$, and
		$$ A\xi \cdot \xi\le\frac{1+(p-2)^+}{2^{p-2}}\,|q|^{p-2}
		$$
		if~$p\in(1,2)$.
		
		These considerations complete the proof of Lemma~\ref{lemma:unifell}.
	\end{proof}
	
	With this, we can now state the following result:
	
	\begin{lem}\label{lemma-second-alternative-dichotomy-improved}
		Let~$a_1>a_0>0$ and
		\begin{equation}\label{poiqwerthfbgihyouoiuo}
		p>\max\left\{\frac{2n}{n+2},1\right\}.\end{equation}
		
		There exist~$\alpha_0\in(0,1]$ and~$\widetilde{C}>0$, depending on~$n$ and~$p$, such that
		for every~$\alpha\in(0,\alpha_0)$ there exist
		\begin{itemize}\item
		$\rho\in(0,1)$, depending on~$n$, $p$ and~$\alpha$,
		\item $\varepsilon_0\in(0,1)$, depending on~$n$, $p$, $a_0$, $a_1$ and~$\alpha$,
		\item $c_0>0$, depending on~$n$, $p$, $a_0$, $a_1$ and~$\alpha$,
		\end{itemize}
		such that, if~$\varepsilon\in(0, \varepsilon_0]$ and~$\sigma\in(0, c_0\varepsilon^P]$, with~$P:=\max\{p,2\}$,
then the following statement holds true.
		
		Let~$u\in W^{1,p}(B_1)$ be such that~$u\ge0$ in~$B_1$ and
		\begin{equation}\label{2.8BIS}
			J_p(u,B_1)\le (1+\sigma) J_p(v,B_1)\end{equation}
		for all~$v\in W^{1,p}(B_1)$ such that~$v=u$ on~$\partial B_1$.

Let
\begin{equation}\label{3.15BIS} a:=\left(\fint_{B_1}\left|\nabla u(x)\right|^p \,dx\right)^{1/p}\end{equation}
		and suppose that
		\begin{equation}\label{316BIS}
		a\in [a_0,a_1].\end{equation}
		
		Assume also that
		\begin{equation}\label{second-alternative-dichotomy}
			\left(\fint_{B_1}\left|\nabla u(x)-q\right|^p \,dx\right)^{1/p}\le \varepsilon a,
		\end{equation}
		for some~$q\in \mathbb{R}^n$ such that
		\begin{equation}\label{estimate-norm-q}
			\frac{a}{8}<|q|\le 2C_0a,
		\end{equation}
		where~$C_0>0$ is the universal constant given by Proposition~\ref{proposition-dichotomy}.
		
		 Then  
		\begin{equation}\label{so3cer5b56859tj45ivnt45-1}
			\left(\fint_{B_\rho}\left|\nabla u(x)-\widetilde{q}\right|^p\,dx\right)^{1/p}\le\rho^{\alpha}\varepsilon a,
		\end{equation}
		with~$\widetilde{q}\in \mathbb{R}^n$ such that
		\begin{equation}\label{so3cer5b56859tj45ivnt45-2}
			\left|q-\widetilde{q}\right|\le \widetilde{C}\varepsilon a.
		\end{equation}
	\end{lem}
	
	The quantity~$\alpha_0$ in the statement of Lemma~\ref{lemma-second-alternative-dichotomy-improved}
	measures the oscillation of~$\nabla u,$ see~\cite{Ur}, \cite{MR709038}, \cite{T} or, specifically for our case, Theorem~2 in~\cite{Man}.
When~$p=2$, we have that~$\alpha_0=1$, according to the full regularity of harmonic functions,
but when~$p\not=2$ in general we only know that~$\alpha_0\in (0,1)$ (but when~$n=2$
a sharp regularity exponent has been determined in~\cite{IwMa}).

	\begin{proof}[Proof of Lemma~\ref{lemma-second-alternative-dichotomy-improved}]
		Let~$\bar{v}$ denote the~$p$-harmonic replacement of~$u$ in~$B_{9/10}$ and
		let~$v$ be defined as
		\begin{equation}\label{defin-of-special-competitor}
			v:=\begin{cases}
				\bar{v}&\mbox{in }B_{9/10},\\
				u&\mbox{in }B_1\setminus B_{9/10}.
			\end{cases}
		\end{equation}
		Then, since~$v=u$ on~$\partial B_1$ and~$v\in W^{1,p}(B_1)$, we
		can use~\eqref{2.8BIS} to see that 
		\[J_p(u,B_1)\le (1+\sigma)J_p(v,B_1).\]
		Consequently,
		\begin{equation*}\begin{split}
				J_p(u,B_{9/10})=\;&J_p(u,B_{1})- J_p(u,B_1\setminus B_{9/10})\\
				\le\;& J_p(v,B_{9/10})
				+J_p(v,B_1\setminus B_{9/10})+\sigma J_p(v,B_1)-J_p(u,B_1\setminus B_{9/10})\\
				=\;&J_p(v,B_{9/10})+J_p(u,B_1\setminus B_{9/10})+\sigma J_p(v,B_1)-J_p(u,B_1\setminus B_{9/10})\\
				=\;&J_p(v,B_{9/10})+\sigma J_p(v,B_1).
		\end{split}\end{equation*}
		Hence, recalling the definition of~$J_p$ in~\eqref{defjp},
		\begin{align*}
			&\int_{B_{9/10}}\left|\nabla u(x)\right|^p\,dx+\left|\left\{u>0\right\}\cap B_{9/10}\right|\le 
			\int_{B_{9/10}}\left|\nabla v(x)\right|^p\,dx+\left|B_{9/10}\right|+\sigma J_p(v,B_1),
		\end{align*}
		which yields that
		\begin{equation}\label{thisineq0987654322}
			\int_{B_{9/10}}\big(\left|\nabla u(x)\right|^p-\left|\nabla v(x)\right|^p\big)\,dx
			\le \left|\left\{u=0\right\}\cap B_{9/10}\right|+\sigma J_p(v,B_1).
		\end{equation}
		Moreover, recalling~\eqref{defin-of-special-competitor}, we point out that~$v$ is the~$p$-harmonic replacement of~$u$ in~$B_{9/10}$, and therefore
		\begin{equation}\begin{split}\label{dweorb849tbv75849}
				J_p(v,B_1) =\;& 
				\int_{B_1}\Big( |\nabla v(x)|^p+\chi_{\{v>0\}}(x)\Big)\,dx\\ \le\;&
				\int_{B_{9/10}}|\nabla v(x)|^p\,dx +\int_{B_1\setminus B_{9/10}}|\nabla v(x)|^p\,dx
				+|B_1|\\ \le\;&
				\int_{B_1}\left|\nabla u\right|^p\,dx+\left|B_1\right|\le a^p+\left|B_1\right|.
		\end{split}\end{equation}
		
		Furthermore, if~$p\ge2$, by~\eqref{second-ineq-lemma-p-harm-repl}
	and~\eqref{thisineq0987654322} we deduce that
		\begin{equation*}\label{second-ineq-condition-almost-minim}
			\int_{B_{9/10}}\left|\nabla u(x)-\nabla v(x)\right|^p\,dx\le
			C \left|\left\{u=0\right\}\cap B_{9/10}\right|+C\sigma J_p(v,B_1),
		\end{equation*}
		for some positive universal constant~$C$. 
		Consequently, exploiting~\eqref{dweorb849tbv75849}, we obtain that
		\begin{equation}\label{third-ineq-condition-almost-minim}
			\int_{B_{9/10}}\left|\nabla u(x)-\nabla v(x)\right|^p\,dx
			\le C \left|\left\{u=0\right\}\cap B_{9/10}\right|+C\sigma(a^p+1),
		\end{equation}
		up to renaming~$C$.
		
If instead~$p<2$, we use~\eqref{first-ineq-lemma-p-harm-repl}, \eqref{thisineq0987654322}
			and~\eqref{dweorb849tbv75849}
			to write that
			\begin{equation}\label{casonuovo}\begin{split}
					& \int_{B_{9/10}}\left|\nabla u(x)-\nabla v(x)\right|^p\,dx\\
					\le\;& C\left(\int_{B_{9/10}}\left(\left|\nabla u(x)\right|^p-\left|\nabla v(x)\right|^p\right)\,dx\right)^{\frac{p}{2}}
					\left(\int_{B_{9/10}}\big(\left|\nabla u(x)\right|+\left|\nabla v(x)\right|\big)^p\,dx\right)^{1-\frac{p}{2}}
					\\ \le\;& C\Big(\left|\left\{u=0\right\}\cap B_{9/10}\right|+\sigma J_p(v,B_1)\Big)^{\frac{p}{2}}
					\left(\int_{B_{9/10}}2^p\left|\nabla u(x)\right|^p\,dx\right)^{1-\frac{p}{2}}	\\ \le\;& C
					\Big(\left|\left\{u=0\right\}\cap B_{9/10}\right|+\sigma(a^p+1)\Big)^{\frac{p}{2}}
					a^{p\left(1-\frac{p}{2}\right)}	
					.\end{split}\end{equation}
		
		Now we claim that
		\begin{equation}\label{estimate-zero-set-almost-minim}
			\left|B_{9/10}\cap \left\{u=0\right\}\right|\le C_1\varepsilon^{p+\delta},
		\end{equation}
		for some~$C_1>0$ and~$\delta>0$. 
		To prove this, we consider the linear function
		\begin{equation}\label{defin-linear-funct}
			\ell(x):= b+q\cdot x,\quad {\mbox{ with }}\quad b:= \fint_{B_1}u(x)\,dx.
		\end{equation}
		We remark that
		\begin{align*}
			&\fint_{B_1}\big(u(x)-\ell(x)\big)\,dx=b-\left(b+\fint_{B_1}q\cdot x\,dx\right)=-\fint_{B_1}q\cdot x\,dx=0.
		\end{align*}
		As a consequence, denoting by
		\[(u-\ell)_{B_1}:=\fint_{B_1}\big(u(x)-\ell(x)\big)\,dx \]
		and using the Poincar\'e inequality, we see that 
		\[\left\|u-\ell-(u-\ell)_{B_1}\right\|_{L^p(B_1)}=
		\left\|u-\ell\right\|_{L^p(B_1)}\le C\left\|\nabla (u-\ell)\right\|_{L^p(B_1)},\]
		for some~$C>0$ universal.
		
		This and~\eqref{second-alternative-dichotomy} entail that
		\begin{equation}\label{woru4itb45v675yt54ty458}
			\fint_{B_1}\left|u(x)-\ell(x)\right|^p\,dx\le 
			C\int_{B_1}|\nabla (u-\ell)(x)|^p\,dx	\le C\varepsilon^pa^p.
		\end{equation}
		We also observe that, since~$u\ge 0,$ it holds that~$\ell^-\le \left|u-\ell\right|$.
		Using this information into~\eqref{woru4itb45v675yt54ty458}, we obtain that 
		\begin{equation}\label{ineq-average-negative-part-linear-funct-to-p}
			\fint_{B_1}(\ell^-(x))^p\,dx\le C\varepsilon^pa^p.
		\end{equation}
		
		Now we claim that, if~$\varepsilon$ is sufficiently small,
		\begin{equation}\label{lower-bound-linear-funct-1}
			\ell\ge c_1a\quad\mbox{in }B_{9/10},
		\end{equation}	
		for some~$c_1>0$. To check this, we argue by contradiction assuming that
		$$ \min_{x\in \overline{B_{9/10}}}\ell(x)<ca$$
		for all~$c>0$. 
		We notice that, for every~$x\in B_{9/10}$,
		$$ -\frac{9|q|}{10}\le \ell(x)-b\le \frac{9|q|}{10}.$$
		and therefore
		$$  ca>\min_{x\in \overline{B_{1/2}}}\ell(x)\ge b-\frac{9|q|}{10}.$$
		This leads to
		\begin{equation}\label{sowerb4ity58deteryhbervt}
			b\le  ca +\frac{9|q|}{10}.\end{equation}
		Now we define
		$${\mathcal{B}}:=\left\{x=-\frac{tq}{|q|}+\eta, \;{\mbox{ for some }} t\in\left[\frac{38}{40},\frac{39}{40}\right] {\mbox{ and }}
		\eta\in B_{1/40} \right\}.$$
		Notice that if~$x\in{\mathcal{B}}$ then
		$$ |x|\le t+|\eta|<\frac{39}{40}+\frac1{40}=1,$$
		and this shows that
		\begin{equation}\label{dketvuibtbuyiuy09876543}
			{\mathcal{B}}\subseteq B_1.\end{equation}
		
		Furthermore, from~\eqref{estimate-norm-q} and~\eqref{sowerb4ity58deteryhbervt} we see that, if~$x\in{\mathcal{B}}$,
		\begin{eqnarray*}
			&& \ell(x)=b - t|q|+\eta\cdot q\le ca +\frac{9|q|}{10}- \frac{38|q|}{40}+\frac{|q|}{40}=
			ca -\left(\frac{38}{40}-\frac{9}{10}-\frac1{40}\right)|q| \\&&\qquad\qquad = ca- \frac{|q|}{40}< ca
			-\frac{a}{320}.
		\end{eqnarray*}
		Hence, taking~$c\in\left(0, \frac{1}{640}\right)$, we infer that
		$$ \ell(x)\le -\frac{a}{640}.$$
		Accordingly, using this and~\eqref{dketvuibtbuyiuy09876543} into~\eqref{ineq-average-negative-part-linear-funct-to-p}, we obtain that
		\begin{eqnarray*}
			&& C\,|B_1|\,\varepsilon^pa^p\ge
			\int_{B_1}(\ell^-(x))^p\,dx\ge \int_{{\mathcal{B}}}(\ell^-(x))^p\,dx\ge
			\int_{{\mathcal{B}}}\left( \frac{a}{80\cdot 2^{(3p+1)/p}}\right)^p\,dx\ge \overline{c} a^p,
		\end{eqnarray*}
		for some positive universal constant~$\overline{c}$.
		This establishes the desired contradiction if~$\e$ is sufficiently small, and thus the proof of~\eqref{lower-bound-linear-funct-1} is complete.
		
Now, we distinguish the three following cases: $p\in (1,n),$ $p=n$ and~$p>n$.
		
Firstly, if~$p<n,$ recalling the Poincar\'e-Sobolev inequality (see e.g. Theorem~2 on page~265 of~\cite{MR1625845}, or Theorem~7.30 in~\cite{GilT}) and applying the same argument of formula~(7.40) in~\cite{GilT}, we obtain that
		\begin{equation}\label{Sobolev-Poincare-ineq}
			\left(\int_{B_1}\left|u(x)-\ell(x)\right|^{p^*}\,dx\right)^{1/p^*}\le C
			\left(\int_{B_1}\left|\nabla \big(u(x)-\ell(x)\big)\right|^{p}\,dx\right)^{1/p}\end{equation}
		for some~$C>0$ universal, where
		\[p^*:= \frac{np}{n-p}.\]		
		Then, by virtue of~\eqref{second-alternative-dichotomy}, \eqref{defin-linear-funct} and~\eqref{Sobolev-Poincare-ineq}, we get that 
		\begin{eqnarray*}
			&&\left(\displaystyle\int_{B_1}\left|u(x)-\ell(x)\right|^{p^*}\,dx\right)^{1/p^*}
			\le C	 \left(\int_{B_1}\left|\nabla \big(u(x)-\ell(x)\big)\right|^{p}\,dx\right)^{1/p}\\
			&&\qquad \qquad = C \left(\int_{B_1}\left|\nabla u(x)-q\right|^{p}\,dx\right)^{1/p}
			\le C\varepsilon a.
		\end{eqnarray*}
		This and~\eqref{lower-bound-linear-funct-1} entail that
		\begin{align*}
			C\varepsilon a\ge \left( \int_{B_{9/10}\cap  \left\{u=0\right\}}\left|\ell(x)
			\right|^{p^*}\,dx\right)^{1/p^*}\ge c_1a\left|B_{9/10}\cap  \left\{u=0\right\}\right|^{1/p^*},
		\end{align*}
		and thus, up to renaming constants,		
		\begin{equation}\label{Sobolev-Poincare-ineq-p-<-n-1}
			\left|B_{9/10}\cap  \left\{u=0\right\}\right|\le C\varepsilon^{p^*}.
		\end{equation}
		Now we notice that
		\[p^*=\frac{np+p^2-p^2}{n-p}=p+\frac{p^2}{n-p}.\]
		Therefore, setting 
		\begin{equation}\label{casepminoren}\delta:=\frac{p^2}{n-p}>0,
		\end{equation}
		we obtain~\eqref{estimate-zero-set-almost-minim} from~\eqref{Sobolev-Poincare-ineq-p-<-n-1}
		in the case~$p<n$.		
		
		If instead~$p>n,$ we recall Theorem~7.17 as well as~(7.12) and~(7.16) in~\cite{GilT}, to see that
		\begin{equation*}
			\sup\limits_{B_1}\left|u-\ell\right|\le 
			C\left( \int_{B_1}\left|\nabla \big(u(x)-\ell(x)\big)\right|^{p}\,dx\right)^{1/p}	\end{equation*}
		for some~$C>0$ universal. 
		{F}rom this, and making again use of~\eqref{second-alternative-dichotomy} and~\eqref{defin-linear-funct},
		we deduce that
		\begin{align*}
			\sup_{B_1}\left|u-\ell\right|\le C\left( \int_{B_1}\left|\nabla (u(x)-\ell(x))\right|^{p}\,dx\right)^{1/p} \le C\e
			a.
		\end{align*}
		As a consequence of this and~\eqref{lower-bound-linear-funct-1}, for all~$x\in B_{9/10}$,
		\begin{eqnarray*}
			C\e a\ge \sup_{B_1}\left|u-\ell\right|\ge \sup_{B_{9/10}}(\ell-u)\ge \ell(x)-u(x)\ge
			c_1a-u(x),
		\end{eqnarray*}
		and thus~$u(x)\ge c_1a-C\e a>0$ for all~$x\in B_{9/10}$, as long as~$\e$ is sufficiently small.
		Accordingly, it follows that
		\begin{equation}\label{si3875bvc9876bv546980987-0987-70-986}
		|B_{9/10}\cap  \left\{u=0\right\}|=0,\end{equation}
		and this gives~\eqref{estimate-zero-set-almost-minim} in the case~$p>n$ as well.
		
It remains to analyze the case~$p=n$. For this purpose,
		we point out that~$u-\ell\in W^{1,n}(B_1)$, and so we can apply Theorem~7.35 and Theorem~7.16 in~\cite{GilT} to obtain that
		\begin{equation*}
			\int_{B_1}\exp\left(\frac{\left|u(x)-\ell(x)\right|}{C_2\left\|\nabla (u-\ell)
				\right\|_{L^n(B_1)}}\right)^{\frac{n}{n-1}}\,dx
			\le C_3\left|B_1\right|,
		\end{equation*}
		where~$C_2$ and~$C_3$ are positive universal constants. 
		
		Also, for all~$x\in B_1$, we have that~$e^x\ge C_4x^{n}$ for some~$C_4>0$, and thus,
		recalling~\eqref{second-alternative-dichotomy} and~\eqref{defin-linear-funct},
		\begin{equation}\label{ineq-case-p-=-n-2}
			\int_{B_1}\left|u(x)-\ell(x)\right|^{\frac{n^2}{n-1}}\,dx\le C
			\left\|\nabla u-q\right\|_{L^n(B_1)}^{\frac{n^2}{n-1}}\le C\varepsilon^{\frac{n^2}{n-1}}
			a^{\frac{n^2}{n-1}},	\end{equation}
		for some~$C>0,$ from which we deduce that, using~\eqref{lower-bound-linear-funct-1} and~\eqref{ineq-case-p-=-n-2},
		\begin{eqnarray*}
			&& C\varepsilon^{\frac{n^2}{n-1}}
			a^{\frac{n^2}{n-1}}\ge
			\int_{B_1}\left|u(x)-\ell(x)\right|^{\frac{n^2}{n-1}}\,dx\ge \int_{B_{9/10}\cap\{u=0\}}
			\left|\ell(x)\right|^{\frac{n^2}{n-1}}\,dx\\&&\qquad\qquad\ge c_1^{\frac{n^2}{n-1}}a^{\frac{n^2}{n-1}}|B_{9/10}\cap\{u=0\}|
			.\end{eqnarray*}
		
		We point out that 
		\[\frac{n^2}{n-1}=\frac{n^2-n+n}{n-1}=n+\frac{n}{n-1},\]
		and therefore, choosing
		$$\delta:=\frac{n}{n-1}>0,$$
		we establish~\eqref{estimate-zero-set-almost-minim} when~$p=n$ as well.
		
		Gathering together~\eqref{third-ineq-condition-almost-minim} and~\eqref{estimate-zero-set-almost-minim}
		we thus conclude that, if~$p\ge2$,
		\begin{equation}\label{forth-ineq-condition-almost-minim}
			\int_{B_{9/10}}\left|\nabla u(x)-\nabla v(x)\right|^p\,dx\le C_1\varepsilon^{p+\delta}
			+C\sigma(a^p+1).
		\end{equation}
Instead, when~$1<p<2$, we make use of~\eqref{casonuovo} and~\eqref{estimate-zero-set-almost-minim}
			and we obtain that	
			\begin{equation}\label{forth-ineq-condition-almost-minimBIS}
				\int_{B_{9/10}}\left|\nabla u(x)-\nabla v(x)\right|^p\,dx\le
				C \Big(C_1\e^{p+\delta}+\sigma(a^p+1)\Big)^{\frac{p}{2}}
				a^{p\left(1-\frac{p}{2}\right)}.
			\end{equation}	
			
		Our aim is now to
		show that, even if~$v-q\cdot x$ is not the~$p$-harmonic replacement of~$u-q\cdot x,$ as in the classical case (see Lemma~2.3 in~\cite{DS}), it satisfies a ``nice'' equation in~$B_{9/10}$.
		
		To this end, we observe that if~$p\ge2$,
		using~\eqref{second-alternative-dichotomy} and~\eqref{forth-ineq-condition-almost-minim},
		\begin{equation}\label{soiw39vh43v5834v5660987654321uytr}\begin{split}
				\int_{B_{9/10}}\left|\nabla v(x)-q\right|^p\,dx \le\;& 2^{p-1}\left(
				\int_{B_{9/10}}\left|\nabla u(x)-q\right|^p\,dx
				+\int_{B_{9/10}}\left|\nabla v(x)-\nabla u(x)\right|^p\,dx\right)\\
				\le\; &  C_2\varepsilon^pa^p+C_1\varepsilon^{p+\delta}+C\sigma(a^p+1),
		\end{split}\end{equation}
		up to renaming constants.
		
If instead~$1<p<2$, from~\eqref{second-alternative-dichotomy}
			and~\eqref{forth-ineq-condition-almost-minimBIS} we have that
			\begin{equation}\label{soiw39vh43v5834v5660987654321uytrBIS}\begin{split}
					\int_{B_{9/10}}\left|\nabla v(x)-q\right|^p\,dx \le\;& 2^{p-1}\left(
					\int_{B_{9/10}}\left|\nabla u(x)-q\right|^p\,dx
					+\int_{B_{9/10}}\left|\nabla v(x)-\nabla u(x)\right|^p\,dx\right)\\
					\le\; &  C_2\varepsilon^pa^p+C \Big(C_1\e^{p+\delta}+\sigma(a^p+1)\Big)^{\frac{p}{2}}
					a^{p\left(1-\frac{p}{2}\right)}.
			\end{split}\end{equation}
		
		Suppose now that~$p\ge2$ and~$\sigma\le c_0\varepsilon^p$, with~$c_0$ to be made precise later.
		Recalling~\eqref{316BIS} and~\eqref{soiw39vh43v5834v5660987654321uytr}, we infer that
		\begin{equation}\label{first-ineq-p-harm-repl-minus-lin-funct}
			\int_{B_{9/10}}\left|\nabla v(x)-q\right|^p\,dx\le C_2\varepsilon^pa^p+C_1\varepsilon^{p+\delta}+Cc_0\varepsilon^p(a^p+1)
			\le C\varepsilon^pa^{p},
		\end{equation}
		up to relabeling~$C>0$.
		
		Similarly, if~$1<p<2$, we take~$\sigma\le c_0\varepsilon^2$, with~$c_0$ to be made precise later. In this case, we deduce from~\eqref{soiw39vh43v5834v5660987654321uytrBIS} that
			\begin{equation}\label{first-ineq-p-harm-repl-minus-lin-functBIS-2000}\begin{split}&
				\int_{B_{9/10}}\left|\nabla v(x)-q\right|^p\,dx\le C_2\varepsilon^pa^p+C \Big(C_1\e^{p+\delta}+c_0 \e^2(a^p+1)\Big)^{\frac{p}{2}}
				a^{p\left(1-\frac{p}{2}\right)}
				\\&\qquad\qquad\le C_1\varepsilon^{(p+\delta)\frac{p}{2}} a^{p }+C_2\varepsilon^{p} a^{p },
			\end{split}\end{equation}
			up to renaming the constants.
			
			We stress that we have not used the assumption on~$p$ given in~\eqref{poiqwerthfbgihyouoiuo} so far. We will use it
		next to reabsorb the term~$\varepsilon^{(p+\delta)\frac{p}{2}} a^{p }$ appearing
		in~\eqref{first-ineq-p-harm-repl-minus-lin-functBIS-2000} into the term~$\varepsilon^{p} a^{p }$.
		More precisely, we point out that
			\begin{equation}\label{slwpv56b987900000}{\mbox{if~$p\in \left(\max\left\{\frac{2n}{n+2},1\right\},\, 2\right)$,\quad then }} (p+\delta)\frac{p}{2}>p.\end{equation}
			Indeed, when~$n=1$ we are in the situation in which~$n=1<p$, and therefore we can take~$\delta=1$
			in~\eqref{estimate-zero-set-almost-minim} (recall also~\eqref{si3875bvc9876bv546980987-0987-70-986}),
			thus obtaining that~$\frac{p+1}{2}>1$, which gives~\eqref{slwpv56b987900000} in this case.
			
			Also, if~$n\ge2$, then we are in the case~$p< n$, and thus, recalling
			the value of~$\delta$ given in~\eqref{casepminoren}, we obtain that
			$$ p+\delta=p+\frac{p^2}{n-p}=\frac{np}{n-p}>2,$$
			which gives~\eqref{slwpv56b987900000}.
			
			As a consequence of~\eqref{first-ineq-p-harm-repl-minus-lin-functBIS-2000}
			and~\eqref{slwpv56b987900000}, we obtain that, if~$p\in \left(\max\left\{\frac{2n}{n+2},1\right\},\, 2\right)$,
				\begin{equation*}
				\int_{B_{9/10}}\left|\nabla v(x)-q\right|^p\,dx\le C\varepsilon^{p} a^{p },\end{equation*}
			for some~$C>0$.
			
			Putting together this and~\eqref{first-ineq-p-harm-repl-minus-lin-funct}, we conclude that,
			if~$p>\max\left\{\frac{2n}{n+2},1\right\}$,
			\begin{equation}\label{first-ineq-p-harm-repl-minus-lin-functBIS-2}
				\int_{B_{9/10}}\left|\nabla v(x)-q\right|^p\,dx\le C\varepsilon^{p} a^{p }.
		\end{equation}

			Furthermore, we claim that, if~$\e$ is sufficiently small,
			for all~$x\in B_{1/2}$,
			\begin{equation}\label{second-ineq-p-harm-repl-minus-lin-funct}
				\left|\nabla v(x)-q\right|\le C (\varepsilon a)^{\nu},
			\end{equation}
			for some~$C>0$ and~$\nu\in(0,1)$. Indeed, suppose by contradiction that
			\begin{equation}\label{soiwer8b95vt 9y589ytoierp}
				\max_{x\in\overline{B_{1/2}}}|\nabla v(x)-q|> C (\varepsilon a)^{\nu}\end{equation}
			for all~$C>0$ and~$\nu\in(0,1)$. We recall that~$v$ is~$p$-harmonic in~$B_{9/10}$ and
			thus the maximum in~\eqref{soiwer8b95vt 9y589ytoierp} is achieved, see e.g.~\cite{MR709038} or~\cite{T}.
			Moreover, let~$\overline{x}\in\overline{B_{1/2}}$
			be such that
			\begin{equation}\label{soiwer8b95vt 9y589ytoierpBIS}
				|\nabla v(\overline{x})-q|=\max_{x\in\overline{B_{1/2}}}|\nabla v(x)-q|
				> C (\varepsilon a)^{\nu}.\end{equation}
			Then, for all~$x\in B_{1/8}(\overline{x})$,
			$$ \big| \nabla v(x) -\nabla v(\overline{x})\big|\le \overline{C}|x-\overline{x}|^{\alpha},$$
			for some~$\overline{C}>0$ and~$\alpha\in(0,1]$, with~$\alpha$ depending only on~$n$ and~$p$,
			and~$\overline{C}$ bounded by a constant depending only on~$n$ and~$p$ times~$\|\nabla v\|_{L^\infty(B_{1/4}(\overline{x}))}$, see Theorem~2 in~\cite{Man}.
			Actually, since, by formula~(3.44) in~\cite{MZ},
$$\|\nabla v\|_{L^\infty(B_{1/4}(\overline{x}))}\le \|\nabla v\|_{L^\infty(B_{3/4})}\le 
\hat{C}\|\nabla v\|_{L^p(B_{9/10})}\le\hat{C}
\|\nabla u\|_{L^p(B_{9/10})}\le \widetilde{C} a_1,$$
for some~$\hat{C}$ and~$\widetilde{C}>0$ depending only on~$n$ and~$p$, we infer that~$\overline{C}$
only depends on~$n$, $p$ and~$a_1$.
			
			As a consequence of this and~\eqref{soiwer8b95vt 9y589ytoierpBIS},
			for all~$x\in B_{\left(\frac {C  (\e a)^{\nu} }{2 \overline{C}} \right)^{1/\alpha}}(\overline{x})$,
			\begin{eqnarray*}&&
				|\nabla v(x)-q|\ge |\nabla v(\overline{x})-q|-\big| (\nabla v(x)-q) -(\nabla v(\overline{x})-q)\big|
				\\&&\qquad\qquad \ge C (\varepsilon a)^{\nu}-\overline{C}|x-\overline{x}|^{\alpha}\ge
				C (\varepsilon a)^{\nu}- \frac{C (\varepsilon a)^{\nu}}2=\frac{C (\varepsilon a)^{\nu}}2.
			\end{eqnarray*}
			Notice also that, if~$\e$ is sufficiently small, then~$B_{\left(\frac {C  (\e a)^{\nu} }{2 \overline{C}} \right)^{1/\alpha}}(\overline{x})\subseteq B_{9/10}$.
			Accordingly,
			\begin{eqnarray*}
				&&\int_{B_{9/10}}\left|\nabla v(x)-q\right|^p\,dx \ge 
				\int_{ B_{\left(\frac {C  (\e a)^{\nu} }{2 \overline{C}} \right)^{1/\alpha}}(\overline{x}) }
				\left|\nabla v(x)-q\right|^p\,dx\ge
				\left(\frac{C (\varepsilon a)^{\nu}}2\right)^p \left|B_{\left(\frac {C  (\e a)^{\nu} }{2 \overline{C}} \right)^{1/\alpha}}(\overline{x})\right|\\&&\qquad\qquad \left(\frac{C (\varepsilon a)^{\nu}}2\right)^p \left(\frac {C  (\e a)^{\nu} }{2 \overline{C}} \right)^{n/\alpha}|B_1|=\frac{ C^{p+\frac{n}{\alpha}} |B_1| }{ 2^{p+\frac{n}{\alpha}}\, \overline{C}^{\frac{n}{\alpha}} }\,
				(\e a)^{\nu\left( p+\frac{n}\alpha\right)}.
			\end{eqnarray*}
			Hence, exploiting~\eqref{first-ineq-p-harm-repl-minus-lin-functBIS-2},
we find that
			$$ C\e^p a^{p }\ge\frac{ C^{p+\frac{n}{\alpha}} |B_1| }{ 2^{p+\frac{n}{\alpha}}\, \overline{C}^{\frac{n}{\alpha}} }\,
			(\e a)^{\nu\left( p+\frac{n}\alpha\right)},$$
			which yields the desired contradiction
			as soon as~$\e$ is chosen sufficiently small and~$\nu\in\left(0,\frac{p}{p+\frac{n}\alpha}\right).$
						The proof of~\eqref{second-ineq-p-harm-repl-minus-lin-funct} is thus complete.
			
			Let us define now the function~$F:\mathbb{R}^n\longrightarrow \mathbb{R}^n$ as~$F(z):=\left|z\right|^{p-2}z$.
			We have that
			\begin{eqnarray*}
				&&\left|\nabla v\right|^{p-2}\nabla v-\left|q\right|^{p-2}q=
				F(\nabla v)-F(\nabla(q\cdot x))=F(\nabla v)-F(q)\\&&\qquad\qquad
				=\int_0^1\frac{d}{dt}F(t\nabla v+(1-t)q)\,dt
				=\int_0^1DF(t\nabla v+(1-t)q)(\nabla v-q)\,dt.
			\end{eqnarray*}
			Hence, taking the divergence of both sides, we obtain that
			\begin{align}\label{unif-ellip-equat-lemma-dich-improved}
				\diverg\big(A (\nabla v-q)\big)=0\quad \mbox{ in }  B_{9/10},
			\end{align}
			with 
			\[
			A(x):=\int_0^1DF\big(t\nabla v(x)+(1-t)q\big)\,dt.
			\]
			Notice that we are in the setting of Lemma~\ref{lemma:unifell} with~$\eta:=\nabla v-q$. Indeed, we point out
			that, in light of~\eqref{estimate-norm-q}
			and~\eqref{second-ineq-p-harm-repl-minus-lin-funct}, for all~$x\in B_{1/4}$,
			$$ |\nabla v(x)-q|\le C (\varepsilon a)^{\nu}\le\frac{a}{16}<\frac{|q|}2,
			$$
			as long as~$\e$ is sufficiently small.
			
			Accordingly, exploiting Lemma~\ref{lemma:unifell} we find that
			$$ \lambda\,|q|^{p-2}\le A\xi \cdot \xi \le\Lambda\, |q|^{p-2},$$
			for all~$\xi\in\partial B_1$,
			for some~$\Lambda\ge\lambda>0$, depending only on~$p$. Recalling~\eqref{estimate-norm-q},
			this gives that
			$$ \lambda\,a^{p-2}\le A\xi \cdot \xi \le\Lambda\, a^{p-2},$$
			up to renaming~$\lambda$ and~$\Lambda$, that now may depend on~$n$ and~$p$.
			
			We point out that we are in the setting of Section~3.2 in~\cite{MZ} (see in particular
			formula~(3.43) there, with~$\lambda$ and~$\Lambda$ replaced by~$\lambda\,a^{p-2}$
			and~$\Lambda\,a^{p-2}$ here, respectively). Hence, we are in the position
			of applying Theorem~3.19 in~\cite{MZ},
			and so we obtain that, for every~$x\in B_{1/2}$,
			\begin{equation}\label{malizimer22}
				\left|\nabla v(x)-q\right|^p\le \sup_{B_{1/8}(x)}\left|\nabla v-q\right|^p\le C\fint_{B_{1/4}(x)}
				\left|\nabla v(y)-q\right|^p\,dy\le C\e^p a^{p},
			\end{equation}
			for some~$C>0$, depending on~$n$ and~$p,$ where we have also
			used~\eqref{first-ineq-p-harm-repl-minus-lin-functBIS-2}.
			
			Consequently, for all~$x\in B_{1/2}$,
			\begin{equation}\label{swewertyuioit5uyi5oyu458wvy4vu56vu00}
				\left|\nabla v(x)-q\right|\le 
				C\varepsilon a,
			\end{equation}
			up to renaming~$C$, depending on~$n$ and~$p$.
			
			Also, denoting~$
			\bar{q}:=\nabla v(0)-q$, we have that 
			\begin{equation}\label{sdo3rv8b95t7438 593456i89087656786}
				|\bar{q}| =|\nabla v(0)-q|\le	
				C\varepsilon a.\end{equation}
			Therefore, from this and~\eqref{swewertyuioit5uyi5oyu458wvy4vu56vu00}, we deduce that, for all~$x\in B_{1/2}$,
			\begin{equation*}
				\left|\nabla v(x)-q-\bar{q}\right|	\le C\varepsilon a.
			\end{equation*}
			As a consequence, exploiting Theorem~2 in~\cite{Man}, we obtain that,
			for every~$\rho\in(0,1/2)$,
			\begin{equation}\label{ineq-average-nabla-p-harm-repl-minus-tilde-q}
				\fint_{B_{\rho}}\left|\nabla v(x)-q-\bar{q}\right|^p\,dx\le 
				\fint_{B_{\rho}}\left(C\left(\frac{\rho}{1/2}\right)^\mu
				\left\|\nabla v-q-\bar{q}\right\|_{L^\infty(B_{1/2})}
				\right)^p\,dx\le C_2\,\rho^{\mu p}\,\varepsilon^p a^{p},
			\end{equation}
			for some~$\mu\in(0,1]$ and~$C_2>0$ depending on~$n$ and~$p$.
			
			Then, if~$p\ge 2$, putting together~\eqref{forth-ineq-condition-almost-minim} and~\eqref{ineq-average-nabla-p-harm-repl-minus-tilde-q}, we arrive at
			\begin{equation}\label{ineq-nabla-almost-minim-minus-tilde-q}\begin{split}&
					\fint_{B_{\rho}}\left|\nabla u(x)-q-\bar{q}\right|^p\,dx\\
					\le\;& 2^{p-1}\left(
					\fint_{B_{\rho}}\left|\nabla u(x)-\nabla v(x)\right|^p\,dx+
					\fint_{B_{\rho}}\left|\nabla v(x)-q-\bar{q}\right|^p\,dx
					\right)\\
					\le\;& 2^{p-1}C_1\varepsilon^{p+\delta}
					\rho^{-n}+2^{p-1}{C}\sigma(a^p+1)\rho^{-n}+2^{p-1}C_2\rho^{\mu p}\varepsilon^pa^{p}.\end{split}
			\end{equation}
			Thus, setting~$\alpha_0:=\mu$ and, for every~$\alpha \in (0,\alpha_0)$,
			\begin{equation}\label{jdiet7ub8v9wq43275bvc687c693--ppe5vb}  
				\rho:= (2^{p+1}C_2)^{\frac{1}{(\alpha- \alpha_0)p}}, \qquad
				\e_0:= \left(\frac{\rho^{\alpha p+n}a_0^p}{2^{p+1}C_1}\right)^{\frac1{\delta}}
\quad{\mbox{and}}\quad
 c_0:=\frac{\rho^{\alpha p+n} a_0^p}{2^{p+1}\, {C}(a_1^p+1)},
\end{equation}
we have that,
for every~$\e\in(0,\e_0]$ and~$\sigma\in(0,c_0\e^p]$, 
			\begin{eqnarray*}
				&&		2^{p-1}C_2\rho^{{\mu} p} \le \frac{1}{4}\rho^{\alpha p},\\
				&&			2^{p-1}C_1\varepsilon^{p+\delta} \rho^{-n}\le\frac{1}{4}\rho^{\alpha p}\varepsilon^p a^p\\
				{\mbox{and }}&& 2^{p-1} {C}\sigma(a^p+1)\rho^{-n}\le\frac{1}{4}\rho^{\alpha p}\varepsilon^p a^p.
			\end{eqnarray*}	
			As a consequence of this and~\eqref{ineq-nabla-almost-minim-minus-tilde-q},
			\begin{align*}
				&\fint_{B_{\rho}}\left|\nabla u(x)-q-\bar{q}\right|^p\,dx\le
				\frac{1}{4}\rho^{\alpha p}\varepsilon^pa^{p }+\frac{1}{4}\rho^{\alpha p}\varepsilon^pa^p+\frac{1}{4}\rho^{\alpha p}\varepsilon^pa^{p }\le \rho^{\alpha p}\varepsilon^pa^{p },
			\end{align*}
			which gives the desired result in~\eqref{so3cer5b56859tj45ivnt45-1} by setting~$\widetilde{q}:=q+\bar{q}$.
			
				Moreover, from~\eqref{sdo3rv8b95t7438 593456i89087656786} we have that~$ |q-\widetilde{q}|=|\bar{q}|\le C\e a$,
			which establishes~\eqref{so3cer5b56859tj45ivnt45-2}.	This completes the proof
			of Lemma~\ref{lemma-second-alternative-dichotomy-improved} when~$p\ge2$.
				
		In the case~$\max\left\{\frac{2n}{n+2},1\right\}<p<2$, we use~\eqref{forth-ineq-condition-almost-minimBIS} and~\eqref{ineq-average-nabla-p-harm-repl-minus-tilde-q} and we see that
				\begin{eqnarray*}&&
					\fint_{B_\rho} |\nabla u(x)-q-\bar{q}|^p\,dx\\
					&\le&2^{p-1}\left(
					\fint_{B_{\rho}}\left|\nabla u(x)-\nabla v(x)\right|^p\,dx+
					\fint_{B_\rho} |\nabla v(x)-q-\bar{q}|^p\,dx
					\right)\\&\le& 2^{p-1} C \Big(C_1\e^{p+\delta}+\sigma(a^p+1)\Big)^{\frac{p}{2}}
					a^{p\left(1-\frac{p}{2}\right)}\rho^{-n}+2^{p-1}C_2\rho^{\mu p}\varepsilon^pa^{p}
					\\&\le& 2^{p-1}C_1\varepsilon^{\frac{(p+\delta)p}2}a^{p\left(1-\frac{p}{2}\right)}
					\rho^{-n}+2^{p-1}{C}\sigma^{\frac{p}2}(a^p+1)^{\frac{p}2}a^{p\left(1-\frac{p}{2}\right)}\rho^{-n}+2^{p-1}C_2\rho^{\mu p}\varepsilon^pa^{p} \\
					&\le& 2^{p-1}C_1\varepsilon^{p+\widetilde\delta}a^{p\left(1-\frac{p}{2}\right)}
					\rho^{-n}+2^{p-1}{C}\sigma^{\frac{p}2}(a^p+1)^{\frac{p}2}a^{p\left(1-\frac{p}{2}\right)}\rho^{-n}+2^{p-1}C_2\rho^{\mu p}\varepsilon^pa^{p},
				\end{eqnarray*}
				where~$\widetilde\delta:=(p+\delta)p/2-p>0$, thanks to~\eqref{slwpv56b987900000} (and up
				to renaming~$C$ and~$C_1$, depending on~$n$ and~$p$).
				
				We set~$\alpha_0:=\mu$ and, for all~$\alpha\in(0,\alpha_0)$, we take~$\rho$
				as in~\eqref{jdiet7ub8v9wq43275bvc687c693--ppe5vb} and   
				\begin{equation*}
				\e_0:= \left(\frac{\rho^{\alpha p+n}a_0^{\frac{p^2}2}}{2^{p+1}C_1}\right)^{\frac1{\widetilde\delta}}
				\quad{\mbox{and}}\quad
				c_0:=\frac{\rho^{2\alpha +\frac{n}{p}} a_0^{p}}{4^{\frac{p+1}p}\, {C}^{\frac2{p}}(a_1^p+1)},\end{equation*}		
				obtaining that, for all~$\e\in(0,\e_0]$ and~$\sigma\in(0,c_0\e^2]$,
				$$ 	\fint_{B_\rho} |\nabla u(x)-q-\bar{q}|^p\,dx\le \rho^{\alpha p}\e^pa^p.$$
Hence, setting~$\widetilde{q}:=q+\bar{q}$, the desired results in~\eqref{so3cer5b56859tj45ivnt45-1} 
and~\eqref{so3cer5b56859tj45ivnt45-2} follow from
this and~\eqref{sdo3rv8b95t7438 593456i89087656786}.

The proof of Lemma~\ref{lemma-second-alternative-dichotomy-improved} is thereby complete.
		\end{proof}
		
		Iterating Lemma~\ref{lemma-second-alternative-dichotomy-improved} we obtain the following
		estimates:
		
		\begin{cor}\label{corollary-lemma-second-alternative-dichotomy-improved}
			Let~$a_1>a_0>0$ and~$p>\max\left\{\frac{2n}{n+2},1\right\}$.
			Let~$u$ be an almost minimizer for~$J_p$ in~$B_1$ (with constant~$\kappa$ and exponent~$\beta$) and
			$$ a:=\left(\fint_{B_1}\left|\nabla u(x)\right|^p \,dx\right)^{1/p}.$$
			Suppose that \begin{equation}\label{A0a1}
			a\in[ a_0,a_1]\end{equation} and that~$u$ satisfies~\eqref{bound-q} and~\eqref{second-alternative-dichotomy}.	
			
			Then there exist~$\varepsilon_0$, $\kappa_0$  and~$\gamma\in(0,1)$, depending on~$n$, $p$, $\beta$, $a_0$
			and~$a_1$, such
			that, for every~$\varepsilon\in(0, \varepsilon_0]$ and~$\kappa\in(0, \kappa_0\varepsilon^P]$,
			with~$P:=\max\{p,2\}$, then
			\begin{equation}\label{first-conclusion-corollary-p-ge-2}
				\left\|u-\ell\right\|_{C^{1,\gamma}(B_{1/2})}\le C\varepsilon a.
			\end{equation}
			 The positive constant~$C$ depends only on~$n$ and~$p$,
			 and~$\ell$ is a linear function of slope~$q$.
			
			Moreover, 
			\begin{equation}\label{second-conclusion-corollary}
				\left\|\nabla u\right\|_{L^{\infty}(B_{1/2})}\le \widetilde{C}a,
			\end{equation}
			with~$\widetilde{C}>0$ depending only on~$n$ and~$p$.
		\end{cor}
		
		\begin{remark}
			We point out that a consequence of~\eqref{first-conclusion-corollary-p-ge-2}
			in Corollary~\ref{corollary-lemma-second-alternative-dichotomy-improved}
			is that, if~$\e$ is sufficiently small, 
			\begin{equation}\label{sowribt4ytkmjnhbgfvcdghjhbgfvdfghhytgrfdewsertyu787654}
				\nabla u\neq 0\quad {\mbox{in }}B_{1/2}.\end{equation}
			Indeed, by~\eqref{first-conclusion-corollary-p-ge-2}, we have that, for all~$x\in B_{1/2}$,
			\[\left|\nabla u(x)-q\right| \le C\varepsilon a,
			\]
			which gives that
			\[
			\left|\nabla u(x)\right| \ge |q|-\left|\nabla u(x)-q\right| \ge \left|q\right|-C\varepsilon a.
			\]
			As a consequence, from~\eqref{bound-q}, we get
			\[\left|\nabla u(x)\right| \ge \frac{a}{4}-C\varepsilon a>0,\]
			as soon as~$\varepsilon$ is sufficiently small, which yields~\eqref{sowribt4ytkmjnhbgfvcdghjhbgfvdfghhytgrfdewsertyu787654}.
			
			Furthermore,
			\begin{equation}\label{kjhgfdsqwertyuio09876543}
				u>0\quad {\mbox{in }} B_{1/2}.\end{equation}
			To check this, we recall that~$u\ge0$ and we suppose by contradiction that
			there exists a point~$x_0\in B_{1/2}$ such that~$u(x_0)=0$. As a consequence, since~$u\in C^{1,\gamma}(B_{1/2})$, we see that~$\nabla u(x_0)=0$, and this contradicts~\eqref{sowribt4ytkmjnhbgfvcdghjhbgfvdfghhytgrfdewsertyu787654}, thus proving~\eqref{kjhgfdsqwertyuio09876543}.
		\end{remark}
		
		In order to prove Corollary~\ref{corollary-lemma-second-alternative-dichotomy-improved}
		we recall the definition of Campanato spaces and a result which we will use in the proof of the corollary,
		see~\cite{G}.
		
		\begin{defin}[Campanato spaces]\label{defin-Campanato-spaces}
			Let~$\Omega$ be a bounded open set in~$\mathbb{R}^n$,
			$1\le p<+\infty$ and~$\lambda \ge 0$. We denote by~$\mathcal{L}^{p,\lambda}(\Omega)$ the space of functions~$u\in L^p(\Omega)$ such that
			\begin{equation}\label{defin-Campanato-seminorm}
				[u]^p_{\mathcal{L}^{p,\lambda}(\Omega)}:= \sup_{\stackrel{x_0\in \Omega}{\rho>0}}\rho^{-\lambda}\int_{\Omega_{x_0,\rho}}\left|u(x)-u_{x_0,\rho}\right|^p\,dx<+\infty,
			\end{equation}
			where
			\[\Omega_{x_0,\rho}:= \Omega\cap B_\rho(x_0)\qquad{\mbox{and}}\qquad
			u_{x_0,\rho}:= \fint_{\Omega_{x_0,\rho}} u(x)\,dx.\]
		\end{defin}
		
		We point out that
		the quantity~$[u]_{\mathcal{L}^{p,\lambda}(\Omega)}$ is a seminorm in~$\mathcal{L}^{p,\lambda}(\Omega)$ and \begin{equation}\label{equivalent-Campanato-seminorm}
			{\mbox{it is equivalent to the quantity }}\quad
			\left(\sup_{\stackrel{x_0\in \Omega}{\rho>0}}\rho^{-\lambda}\inf_{\xi\in \mathbb{R}^n}\int_{\Omega_{x_0,\rho}}\left|u(x)-\xi\right|^p\,dx\right)^{1/p}.\end{equation}
		We also define the norm in~$\mathcal{L}^{p,\lambda}(\Omega)$ as
		\begin{equation}\label{defin-norm-L-p-lambda}
			\left\|u\right\|_{\mathcal{L}^{p,\lambda}(\Omega)}
			:=\left\|u\right\|_{L^p(\Omega)}+[u]_{\mathcal{L}^{p,\lambda}(\Omega)}.
		\end{equation}	
		
		\begin{thm}\label{theo-isomorph-Campanato-spaces-Holder-spaces}
			Let~$x\in\R^n$ and~$r>0$. Let~$1\le p<+\infty$ and~$\lambda \ge 0$.
			Suppose that~$n<\lambda\le n+p$.
			
			Then, the space~$\mathcal{L}^{p,\lambda}(B_r(x))$ is isomorphic to~$C^{0,\alpha}(B_r(x)),$ with~$\alpha:=\frac{\lambda-n}{p}.$
		\end{thm}
		
		We also point out the following scaling property
		of almost minimizers:
		
		\begin{lem}\label{remark-rescaling-almost-minimizer-1}
			Let~$u$ be an almost minimizer for~$J_p$ in~$B_1$ with constant~$\kappa$ and exponent~$\beta$. 
			For any~$r\in(0,1)$, let
			\begin{equation}\label{defin-rescaling}
				u_r(x):=\frac{u(rx)}{r}.
			\end{equation} 
			
			Then, $u_r$ is			
		an almost minimizer for~$J_p$ in~$B_{1/r}$ with constant~$\kappa r^\beta$ and exponent~$\beta$, namely
			\begin{equation}\label{THdefin-rescaling} J_p(u_r,B_\varrho(x_0))\leq(1+\kappa r^\beta \varrho^\beta)J_p(v,B_\varrho(x_0)),
		\end{equation}
		for every ball~$B_\varrho(x_0)$ such that~$\overline{B_\varrho(x_0)}\subset B_{1/r}$ and for every~$v\in W^{1,p}(B_\varrho(x_0))$ such that~$v=u_r$ on~$\partial B_\varrho(x_0)$ in the sense of the trace.
			\end{lem}
		
		\begin{proof} 
			By definition of almost minimizers, we know that
			\begin{equation}\label{condition-almost-minimality}
				J_p(u,B_\vartheta(y_0))\leq(1+\kappa \vartheta^\beta)J_p(w,B_\vartheta(y_0))
			\end{equation}
			for every ball~$B_\vartheta(y_0)$ such that~$\overline{B_\vartheta(y_0)}\subset B_{1/r}$ and for every~$w\in W^{1,p}(B_\vartheta(y_0))$ such that~$w=u$ on~$\partial B_\vartheta(y_0)$ in the sense of the trace.
			
			Now, given~$x_0\in B_{1/r}$, we take~$\varrho$ and~$v$ as in the statement of Lemma~\ref{remark-rescaling-almost-minimizer-1} and we define
			\begin{equation}\label{defin-v_r}
				w(x):= rv\left(\frac{x}{r}\right).
			\end{equation}	
			Then, using the notation~$y_0:=r x_0$ and~$\vartheta:=r\varrho$,
			for all~$x\in\partial B_\vartheta(y_0)$, we have that~$\frac{x}r\in\partial B_\varrho(x_0)$ and therefore, in the sense of the trace,
			\begin{equation*}
				w(x)=rv\left(\frac{x}{r}\right)  =ru_r\left(\frac{x}{r}\right)=u(x).
			\end{equation*}
			Accordingly, we can use~$w$ as a competitor for~$u$ in~\eqref{condition-almost-minimality},
			thus obtaining that
\begin{equation}\label{00PPJnd2}
	\int_{B_{r\varrho}(y_0)}\Big(\left|\nabla u(y)\right|^p+\chi_{\left\{u>0\right\}}(y)\Big)\,dy
\leq(1+\kappa r^\beta\varrho^\beta)
\int_{B_{r\varrho}(y_0)}\Big(\left|\nabla w(y)\right|^p+\chi_{\left\{w>0\right\}}(y)\Big)\,dy.
\end{equation}			

Furthermore, using, consistently with~\eqref{defin-rescaling}, the notation~$w_r(x):=\frac{w(rx)}{r}$, with the change of variable~$x:=\frac{y}r$ we see that
\begin{equation}\label{00PPJnd}\begin{split}&
\int_{B_{r\varrho}(y_0)}\Big(\left|\nabla w(y)\right|^p+\chi_{\left\{w>0\right\}}(y)\Big)\,dy= {r^n}
\int_{B_{\varrho}(x_0)}\Big(\left|\nabla w(rx)\right|^p+\chi_{\left\{w>0\right\}}(rx)\Big)\,dx\\
&\qquad \qquad=r^n\int_{B_{\varrho}(x_0)}\Big(\left|\nabla w_r(x)\right|^p+\chi_{\left\{w_r>0\right\}}(x)\Big)\,dx
\end{split}\end{equation}
and a similar identity holds true with~$u$ and~$u_r$ replacing~$w$ and~$w_r$.

Also, recalling~\eqref{defin-v_r},
we observe that~$v=w_r$. Plugging this information and~\eqref{00PPJnd}
into~\eqref{00PPJnd2}, we obtain the desired result in~\eqref{THdefin-rescaling}.\end{proof}
		
		\begin{proof}[Proof of Corollary~\ref{corollary-lemma-second-alternative-dichotomy-improved}]
Up to scaling, we suppose that
\begin{equation}\label{KSM-kmb2}
{\mbox{$u$ is an almost minimizer for~$J_p$ in~$B_2$ (with constant~$\kappa$ and exponent~$\beta$).}}
\end{equation}

We prove that we can iterate Lemma~\ref{lemma-second-alternative-dichotomy-improved} indefinitely
with~$\alpha:=\min\left\{\frac{\alpha_0}2,\frac{\beta}{P}\right\}$,
where~$P:=\max\{p,2\}$ and~$\alpha_0$ is given in Lemma~\ref{lemma-second-alternative-dichotomy-improved}.

More precisely, we claim that, for all~$ k\ge 0$, there exists~$q_k\in\R^n$ such that
			\begin{equation}\label{jdietrube8t5746v9854-bis}\begin{split}&
|q_k|\in\left[ \frac{a}4-\frac{\widetilde C\e a\left( 1-\rho^{k\alpha}\right)}{1-\rho^{\alpha}}, \,C_0a+\frac{\widetilde C\e a\left( 1-\rho^{k\alpha}\right)}{1-\rho^{\alpha}}\right]
,\\&
				\left(\fint_{B_{\rho^k}}\left|\nabla u(x)-q_{k}\right|^p\,dx\right)^{1/p}\le\rho^{{k\alpha}}\varepsilon a\\{\mbox{and }}\quad&
				\left(\fint_{B_{\rho^k}}\left|\nabla u(x)\right|^p\,dx\right)^{1/p}\in
				\left[ \frac{|q|}2,2|q|\right],
\end{split}
			\end{equation} 
			where~$\rho\in(0,1)$, $C_0>0$ and~$\widetilde C>0$ are universal quantities.

To prove this, we argue by induction. When~$k=0$, we pick~$q_0:=q$. Then, in this case, the desired claims in~\eqref{jdietrube8t5746v9854-bis} follow from~\eqref{bound-q}, \eqref{3.15BIS} and~\eqref{second-alternative-dichotomy}.

Now we perform the inductive step by assuming that~\eqref{jdietrube8t5746v9854-bis} is satisfied
for~$k$ and we establish the claim for~$k+1$. To this end, we let~$r:=\rho^k$ and utilize the rescaling~$u_r$ defined in~\eqref{defin-rescaling}. In this setting, the inductive assumption gives that
$$ \left(\fint_{B_{1}}\left|\nabla u_r(x)-q_{k}\right|^p\,dx\right)^{1/p}\le r^{\alpha}\varepsilon a=\varepsilon_k a,$$
with~$\varepsilon_k:=r^{\alpha}\varepsilon=\rho^{k\alpha}\varepsilon$.

Notice that the inductive assumption also yields~\eqref{estimate-norm-q}, as soon as~$\varepsilon$ is chosen conveniently
small. Therefore, thanks to Lemma~\ref{remark-rescaling-almost-minimizer-1} as well,
we are in position of using Lemma~\ref{lemma-second-alternative-dichotomy-improved}
on the function~$u_r$ with~$\sigma:=\kappa r^\beta$.
We stress that the structural condition~$\sigma\le c_o\varepsilon^P$ in Lemma~\ref{lemma-second-alternative-dichotomy-improved} translates here into~$\kappa\le c_0\varepsilon^P$, which is precisely the requirement in the statement of
Corollary~\ref{corollary-lemma-second-alternative-dichotomy-improved} (by taking~$\kappa_0$ there less than or equal to~$c_0$).
In this way, we deduce from~\eqref{so3cer5b56859tj45ivnt45-1} and~\eqref{so3cer5b56859tj45ivnt45-2}
that there exists~$q_{k+1}\in\R^n$ such that
		\begin{equation*}\begin{split}
			\left(\fint_{B_\rho}\left|\nabla u_r(x)-q_{k+1}\right|^p\,dx\right)^{1/p}\le\rho^{\alpha}\varepsilon_k a\qquad{\mbox{and}}\qquad
\left|q_k-{q}_{k+1}\right|\le \widetilde{C}\varepsilon_k a.	
		\end{split}\end{equation*}
Scaling back, we find that
$$ 	\left(\fint_{B_{\rho^{k+1}}}\left|\nabla u(x)-q_{k+1}\right|^p\,dx\right)^{1/p}\le\rho^\alpha\rho^{{k\alpha}}\varepsilon a=
\rho^{{(k+1)\alpha}}\varepsilon a.$$
Furthermore,
\begin{eqnarray*}&& |q_{k+1}|\le |q_k-q_{k+1}|+|q_k|\le\widetilde{C}\varepsilon_k a+
C_0a+\frac{\widetilde C\e a\left( 1-\rho^{k\alpha}\right)}{1-\rho^{\alpha}}\\&&\qquad
=C_0a+\frac{\widetilde C\e a\left( 1-\rho^{k\alpha}\right)}{1-\rho^{\alpha}}+\widetilde{C}\rho^{k\alpha}\varepsilon a=
C_0a+\frac{\widetilde C\e a\left( 1-\rho^{(k+1)\alpha}\right)}{1-\rho^{\alpha}}
\end{eqnarray*}
and
\begin{eqnarray*}&& |q_{k+1}|\ge |q_k|-|q_k-q_{k+1}|\ge
\frac{a}4-\frac{\widetilde C\e a\left( 1-\rho^{k\alpha}\right)}{1-\rho^{\alpha}}-
\widetilde{C}\varepsilon_k a\\&&\qquad
=\frac{a}4-\frac{\widetilde C\e a\left( 1-\rho^{k\alpha}\right)}{1-\rho^{\alpha}}-
\widetilde{C}\rho^{k\alpha}\varepsilon  a=\frac{a}4-\frac{\widetilde C\e a\left( 1-\rho^{(k+1)\alpha}\right)}{1-\rho^{\alpha}}.\end{eqnarray*}
In addition,
\begin{eqnarray*}&&
\left|\left(\fint_{B_{\rho^{k+1}}}\left|\nabla u(x)\right|^p\,dx\right)^{1/p}-q_{k+1}\right|=
\frac{1}{|B_{\rho^{k+1}}|^{1/p}}
\left|\left(\int_{B_{\rho^{k+1}}}\left|\nabla u(x)\right|^p\,dx\right)^{1/p}-q_{k+1}|B_{\rho^{k+1}}|^{1/p}\right|\\&&\qquad=
\frac{1}{|B_{\rho^{k+1}}|^{1/p}}
\left| \left\|\nabla u\right\|_{L^p(B_{\rho^{k+1}})}-\left\|q_{k+1}\right\|_{L^p(B_{\rho^{k+1}})}\right|\le
\frac{1}{|B_{\rho^{k+1}}|^{1/p}}
\left\|\nabla u-q_{k+1}\right\|_{L^p(B_{\rho^{k+1}})}\\
&&\qquad=\left(\fint_{B_{\rho^{k+1}}}\left|\nabla u(x)-q_{k+1}\right|^p\,dx\right)^{1/p}\le\e a,
\end{eqnarray*}
which yields that
\begin{eqnarray*}&&
\left|\left(\fint_{B_{\rho^{k+1}}}\left|\nabla u(x)\right|^p\,dx\right)^{1/p}-q\right|=
\left|\left(\fint_{B_{\rho^{k+1}}}\left|\nabla u(x)\right|^p\,dx\right)^{1/p}-q_0\right|\\&&\qquad\le
\left|\left(\fint_{B_{\rho^{k+1}}}\left|\nabla u(x)\right|^p\,dx\right)^{1/p}-q_{k+1}\right|+\sum_{j=0}^k|q_{j+1}-q_j|\\&&\qquad\le
\e a+\widetilde{C}a\sum_{j=0}^k\varepsilon_j
\le \e a+\widetilde{C}\varepsilon a\sum_{j=0}^{+\infty}\rho^{j\alpha}=
\left(1+\frac{\widetilde{C}}{1-\rho^\alpha}\right)\e a\\&&\qquad\le
\left(1+\frac{\widetilde{C}}{1-\rho^\alpha}\right)\e |q|\le\frac{|q|}2.
\end{eqnarray*}
These observations conclude the proof of the inductive step and establish~\eqref{jdietrube8t5746v9854-bis}.

By scaling the relations in~\eqref{jdietrube8t5746v9854-bis}
(and recalling~\eqref{KSM-kmb2}), we have that, for all~$x_0\in B_{3/4}$,
\begin{equation}\label{KSd-fqqun8ieI} \left(\fint_{B_{\rho^k}(x_0)}\left|\nabla u(x)-q_{k,x_0}\right|^p\,dx\right)^{1/p}\le C\rho^{k\alpha}\varepsilon a,\end{equation}
for some positive constant~$C$, depending only on~$n$ and~$p$, and suitable points~$q_{k,x_0}$ such that
$$|q_{k,x_0}|\in\left[ \frac{a}4-\frac{\widetilde C\e a\left( 1-\rho^{k\alpha}\right)}{1-\rho^{\alpha}}, \,C_0a+\frac{\widetilde C\e a\left( 1-\rho^{k\alpha}\right)}{1-\rho^{\alpha}}\right].$$

We now want to exploit the Campanato estimate in Theorem~\ref{theo-isomorph-Campanato-spaces-Holder-spaces},
here applied to the gradient of~$u$ minus~$q$,
and used with~$\Omega:=B_{1/2}$ and~$\lambda:=n+\alpha p$
(note that, with this choice, $\lambda\in(n,n+p]$ and we are therefore under
the structural assumptions of Theorem~\ref{theo-isomorph-Campanato-spaces-Holder-spaces}). For this purpose, we claim that, for all~$x_0\in B_{3/4}$ and~$\tau>0$,
\begin{equation}\label{80uefrepsilon-a}
\left(\tau^{-\lambda}\inf_{\xi\in \mathbb{R}^n}\int_{B_\tau(x_0)\cap B_{1/2}}\left|\nabla u(x)-q-\xi\right|^p\,dx\right)^{1/p}\le C\e a,
\end{equation}
up to renaming~$C>0$.

To prove this, we distinguish two cases, either~$\tau\ge1$ or~$\tau\in(0,1)$.

If~$\tau\ge1$, we use~\eqref{second-alternative-dichotomy}
and we get that
\begin{eqnarray*}&&
\left(\tau^{-\lambda}\inf_{\xi\in \mathbb{R}^n}\int_{B_\tau(x_0)\cap B_{1/2}}\left|\nabla u(x)-q-\xi\right|^p\,dx\right)^{1/p}\le
\left(\int_{B_{1/2}}\left|\nabla u(x)-q\right|^p\,dx\right)^{1/p}
\\&&\qquad\le\left(|B_1|\fint_{B_{1}}\left|\nabla u(x)-q\right|^p\,dx\right)^{1/p}
\le |B_1|^{1/p}\,\varepsilon a,
\end{eqnarray*}
which gives~\eqref{80uefrepsilon-a} in this case.

If instead~$\tau\in(0,1)$, we pick~$k\in\N$ such that~$\rho^{k+1}<\tau\le\rho^k$ and we make use of~\eqref{KSd-fqqun8ieI} to infer that
\begin{eqnarray*}&&
\left(\tau^{-\lambda}\inf_{\xi\in \mathbb{R}^n}\int_{B_\tau(x_0)\cap B_{1/2}}\left|\nabla u(x)-q-\xi\right|^p\,dx\right)^{1/p}\le
\left(\rho^{-\lambda(k+1)}\int_{B_{\rho^k}(x_0)}\left|\nabla u(x)-q_{k,x_0}\right|^p\,dx\right)^{1/p}\\&&\qquad=
\left(\rho^{-\lambda(k+1)+kn}\, |B_1|\fint_{B_{\rho^k}(x_0)}\left|\nabla u(x)-q_{k,x_0}\right|^p\,dx\right)^{1/p}\\&&\qquad
\le C|B_1|^{1/p}\,\rho^{[-(n+\alpha p)(k+1)+kn+k\alpha p]/p}\,\varepsilon a\\&&\qquad\le C|B_1|^{1/p}\,\rho^{[-n-\alpha p]/p}\,\varepsilon a,
\end{eqnarray*}
which gives~\eqref{80uefrepsilon-a} up to renaming~$C$. The proof of~\eqref{80uefrepsilon-a}
is thereby complete.

By~\eqref{equivalent-Campanato-seminorm} and~\eqref{80uefrepsilon-a}, we obtain that
$$ [\nabla u-q]_{\mathcal{L}^{p,\lambda}(B_{1/2})}\le C\varepsilon a,$$
up to renaming~$C$. Since also
$$ \left\|\nabla u-q\right\|_{L^p(B_{1/2})}\le C\varepsilon a,$$
thanks to~\eqref{second-alternative-dichotomy},
we deduce from~\eqref{defin-norm-L-p-lambda} that~$\left\|\nabla u-q\right\|_{\mathcal{L}^{p,\lambda}(B_{1/2})}\le
C\varepsilon a$, up to renaming~$C$. Then, in view of Theorem~\ref{theo-isomorph-Campanato-spaces-Holder-spaces},
we find that
\begin{equation}\label{mIPKjkyhneY7ujdf}
\left\|\nabla u-q\right\|_{C^{0,\gamma}(B_{1/2})}\le
C\varepsilon a,\end{equation} with~$\gamma=\frac{\lambda-n}{p}=\alpha$.

Now we define~$\ell(x):=u(0)+q\cdot x$. For all~$x\in B_{1/2}$ we have that
$$ |u(x)-\ell(x)|=|u(x)-u(0)-q\cdot x|=
\left|\int_0^1 \big(\nabla u(tx)-q\big)\cdot x\,dt
\right|\le C\varepsilon a,
$$
thanks to~\eqref{mIPKjkyhneY7ujdf}, up to renaming~$C$, and therefore~$\|u-\ell\|_{L^\infty(B_{1/2})}\le C\varepsilon a$.
This and~\eqref{mIPKjkyhneY7ujdf} establish~\eqref{first-conclusion-corollary-p-ge-2}, as desired.
\end{proof}

 \section{Lipschitz continuity of almost minimizers and proof
of Theorem~\ref{theor-Lipsch-contin-alm-minim}}\label{sec:lip1}

We are now in position of establishing the Lipschitz regularity result in
Theorem~\ref{theor-Lipsch-contin-alm-minim}. 
		
		\begin{proof}[Proof of Theorem~\ref{theor-Lipsch-contin-alm-minim}]
			In the light of Lemma~\ref{remark-rescaling-almost-minimizer-1}, up to a rescaling,
			we can assume that~$u$ is an almost minimizer with constant \begin{equation}\label{1-e2LS}
			\widetilde{\kappa}:=\kappa s^{\beta},\end{equation} which can be made arbitrarily small by an appropriate choice of~$s>0$.
			
			We let~$P:=\max\{p,2\}$. Let also~$\alpha_0\in(0,1]$ be the structural constant given by Lemma~\ref{lemma-second-alternative-dichotomy-improved} and define
			$$ \alpha:=\frac12\min\left\{\alpha_0,\frac\beta{P}\right\}.$$
We also consider~$\e_0$ as given by Proposition~\ref{proposition-dichotomy} and
take~$\eta\in(0,1)$ and~$M\ge1$ as in Proposition~\ref{proposition-dichotomy}
(corresponding here to choice~$\varepsilon:=\varepsilon_0/2$).

Let us define 
			\begin{equation}\label{average-nabla-u-to-^p-in-B_tau-to-1/p}
				a(\tau):= \bigg(\fint_{B_{\tau}}\left|\nabla u(x)\right|^p\,dx\bigg)^{1/p}.
			\end{equation}
			We claim that, for every~$r\in(0,\eta]$,
			\begin{equation}\label{ineq-a(r)-r<1}
a(r)\le C(M,\eta)(1+a(1)),
\end{equation}
for some~$C(M,\eta)>0$, possibly depending on~$n$ and~$p$ as well.

To prove this, we consider the set~${\mathcal{K}}\subseteq\N=\{0,1,2,\dots\}$ containing all the integers~$k\in\N$ such that
			\begin{equation}\label{ineq-a(eta^k)}
				a(\eta^k)\le C(\eta)M+2^{-k}a(1),
			\end{equation}
where
\begin{equation}\label{paJOAKHSVEW9UTIG49-WROITGJH-2}
			C(\eta):= 2\eta^{-n/p} . \end{equation}
			We stress that for~$k=0$ formula~\eqref{ineq-a(eta^k)} is clearly true, hence
			\begin{equation}\label{loaskxZZPL}
			0\in{\mathcal{K}}\ne\varnothing.\end{equation}
			We then distinguish two cases, namely whether~\eqref{ineq-a(eta^k)} holds for every~$k$ (i.e. ${\mathcal{K}}=\N$) or not (i.e. ${\mathcal{K}}\subsetneqq\N$). 
			
			To start with, let us assume that we are in the first situation. Thus, for every~$r\in(0,\eta],$ we pick~$k_0\in\N={\mathcal{K}}$ such that~$\eta^{k_0+1}<r\le \eta^{k_0}.$ Hence, according to~\eqref{average-nabla-u-to-^p-in-B_tau-to-1/p} and~\eqref{ineq-a(eta^k)}, we get that
			\begin{equation*}\begin{split}
				&a(r)\le \bigg(\frac{1}{|B_1|\,\eta^{(k_0+1)n} }\int_{B_{\eta^{k_0}}}\left|\nabla u(x)\right|^p\,dx\bigg)^{1/p}=\eta^{-n/p}a(\eta^{k_0})\le \eta^{-n/p}( C(\eta)M+a(1))\\
				&\qquad\qquad\le \eta^{-n/p}\max\left\{ C(\eta)M,1\right\}(1+a(1))\le C(M,\eta)(1+a(1)),
\end{split}\end{equation*}
provided that~$C(M,\eta)
\ge\eta^{-n/p}\max\left\{ C(\eta)M,1\right\}$.
The proof of~\eqref{ineq-a(r)-r<1} is thereby complete in this case.
			
			We now consider the second case, that is the one in which~${\mathcal{K}}\subsetneqq\N$.
Then, by~\eqref{loaskxZZPL}, there exists~$k_0\in\N$ such that~$\{0,\dots,k_0\}\in{\mathcal{K}}$
and
\begin{equation}\label{pjqsodsckq-=i0dw0pfy9qdhiwcs}
k_0+1\not\in{\mathcal{K}}.\end{equation} We notice that, by~\eqref{average-nabla-u-to-^p-in-B_tau-to-1/p}
and~\eqref{paJOAKHSVEW9UTIG49-WROITGJH-2},
$$\eta^{-n/p}M< C(\eta)M\le C(\eta)M+2^{-(k_0+1)}a(1)<
a(\eta^{k_0+1})\le\eta^{-n/p}a(\eta^{k_0})$$
and therefore
\begin{equation}\label{DALS:iskcd0}a(\eta^{k_0})> M.
\end{equation}
Furthermore, from~\eqref{pjqsodsckq-=i0dw0pfy9qdhiwcs},
\begin{equation}\label{DALS:iskcd}
a(\eta^{k_0+1})>C(\eta)M+2^{-(k_0+1)}a(1)\ge\frac{C(\eta)M+2^{-k_0}a(1)}{2}\ge \frac{a(\eta^{k_0})}2.
\end{equation}

			Now, we claim that
			\begin{equation}\label{second-alternative-dichotomy-B_eta^k}\begin{split}&
				\left(\fint_{B_{\eta^{k_0+1}}}\left|\nabla u(x)-q\right|^p \,dx\right)^{1/p}\le \varepsilon a(\eta^{k_0}),
			\\&{\mbox{for some~$q\in\R^n$ such that }}\quad
			\frac{a(\eta^{k_0})}{4}<\left|q\right|\le C_0a(\eta^{k_0}),\end{split}\end{equation}
			being~$C_0$ the constant given by Proposition~\ref{proposition-dichotomy}.
			
To prove this, we apply the dichotomy result in Proposition~\ref{proposition-dichotomy} rescaled in the ball~$B_{\eta^{k_0}}$.
In this way, we deduce from~\eqref{first-conclusion-dichotomy}, \eqref{second-conclusion-dichotomy} 
and~\eqref{bound-q} that~\eqref{second-alternative-dichotomy-B_eta^k} holds true, unless
\[ a(\eta^{k_0+1})=\left(\fint_{B_{\eta^{k_0+1}}}\left|\nabla u(x)\right|^p \,dx\right)^{1/p}\le \frac{ a(\eta^{k_0}) }2, \]
in contradiction with~\eqref{DALS:iskcd}.
			This ends the proof of~\eqref{second-alternative-dichotomy-B_eta^k}.
			
Now we apply Corollary~\ref{corollary-lemma-second-alternative-dichotomy-improved} rescaled in the ball~$B_{\eta^{k_0+1}}$:
namely, we use here Corollary~\ref{corollary-lemma-second-alternative-dichotomy-improved} with~$B_1$ replaced
by~$B_{\eta^{k_0+1}}$ and~$a$ replaced by~$a(\eta^{k_0+1})$. 
To this end, we need to verify that the assumptions of Corollary~\ref{corollary-lemma-second-alternative-dichotomy-improved} are
fulfilled in this rescaled situation. Specifically, we note that, by~\eqref{DALS:iskcd0} and~\eqref{DALS:iskcd},
$$ a(\eta^{k_0+1})\ge\frac{M}{2}.$$

Also, since~$k_0\in{\mathcal{K}}$,
$$ a(\eta^{k_0+1})\le\eta^{-n/p}a(\eta^{k_0})\le \eta^{-n/p}\big(C(\eta)M+2^{-k_0}a(1)
			\big)\le \eta^{-n/p}\big(C(\eta)M+a(1)\big).$$
			These observations give that~\eqref{A0a1} is satisfied in this rescaled setting with
			\begin{equation}\label{speriamobene}
			a_0:=\frac{M}2\qquad {\mbox{ and }}\qquad a_1:=\eta^{-n/p}\big(C(\eta)M+a(1)\big).\end{equation}
			
			Moreover, by~\eqref{DALS:iskcd} and~\eqref{second-alternative-dichotomy-B_eta^k},
$$\left(\fint_{B_{\eta^{k_0+1}}}\left|\nabla u(x)-q\right|^p \,dx\right)^{1/p}\le 2\varepsilon a(\eta^{k_0+1}),
$$ showing that~\eqref{second-alternative-dichotomy} is satisfied (here with~$2\varepsilon$ instead of~$\varepsilon$).

We also deduce from~\eqref{DALS:iskcd} and~\eqref{second-alternative-dichotomy-B_eta^k} that
$$			\frac{\eta^{n/p}a(\eta^{k_0+1})}{4}\le\frac{a(\eta^{k_0})}{4}<\left|q\right|\le C_0a(\eta^{k_0})\le2C_0a(\eta^{k_0+1}),
$$
which gives that~\eqref{bound-q} is satisfied here (though with different structural constants).

We can thereby exploit Corollary~\ref{corollary-lemma-second-alternative-dichotomy-improved}
in a rescaled version
to obtain the existence of~$\e_0$ (possibly different from the one given by Proposition~\ref{proposition-dichotomy})
and~$\kappa_0$, depending on~$n$, $p$, $\beta$, $a_0$ and~$a_1$,
such that, if
\begin{equation}\label{condaggforseokroetu}
\widetilde\kappa\in(0,\kappa_0\e_0^P],\end{equation} 
we have that
\begin{equation}\label{swqevu3576435432647328528765uhgfhdf}
				\left\|\nabla u\right\|_{L^{\infty} (B_{\eta^{k_0+1}/2} )}\le \bar{C}a(\eta^{k_0}),
			\end{equation}
			for some structural constant~$\bar{C}>0$.
			
We remark that condition~\eqref{condaggforseokroetu} is satisfied by taking~$s$ in~\eqref{1-e2LS}
sufficiently small, namely taking~$s:= \left(\frac{\kappa_0 \e_0^P}{2\kappa}\right)^{\frac1{\beta}}$.
Notice in particular that
\begin{equation}\label{noticeinpartlsworeteuyti8787686}
{\mbox{$s$ depends on~$n$, $p$, $\kappa$,
$\beta$ and~$\|\nabla u\|_{L^p(B_1)}$,}}\end{equation} due to~\eqref{speriamobene}.	

As a result of~\eqref{swqevu3576435432647328528765uhgfhdf}, for all~$r\in\left(0,\frac{\eta^{k_0+1}}2\right)$,
\begin{equation}\label{ineq-a(r)-r-le-eta^k_0/2}\begin{split}&
a(r)=\left( \frac1{|B_r|}\int_{B_r}|\nabla u(x)|^p\,dx\right)^{\frac1p}\le \bar{C}a(\eta^{k_0})\le
\bar{C}\big( C(\eta)M+2^{-k_0}a(1)\big)\\&\qquad\qquad\le\bar{C}\big( C(\eta)M+a(1)\big)\le C(M,\eta)(1+a(1)),
\end{split}\end{equation}
as long as~$C(M,\eta)\ge\bar{C}(C(\eta)M+1)$.

Moreover, if~$r\in\left[\frac{\eta^{k_0+1}}2,\eta\right)$ then we take~$k_r\in\N$ such that~$\eta^{k_r+1}<r\le\eta^{k_r}$.
Note that
$$\frac1{\eta^{k_r}}\le\frac1r\le\frac{2}{\eta^{k_0+1}},$$
whence
\begin{equation}\label{L23tSx32}
 k_r \le k_0+C_\star,
\end{equation}
where~$C_\star:=1+\frac{\log2}{\log(1/\eta)}$.

We now distinguish two cases: if~$k_r\in\{0,\dots,k_0\}$ then~$k_r\in{\mathcal{K}}$ and therefore
\begin{equation*}
a(\eta^{k_r})\le C(\eta)M+2^{-k_r}a(1).
\end{equation*}
{F}rom this we obtain that
\begin{equation}\label{kpd012}
\begin{split}&
a(r)\le\left( \frac1{|B_{\eta^{k_r+1}}|}\int_{B_{\eta^{k_r}}}|\nabla u(x)|^p\,dx\right)^{\frac1p}
=\eta^{-n/p}a(\eta^{k_r})\\&\qquad\qquad\le\eta^{-n/p}\big( C(\eta)M+2^{-k_r}a(1)\big)\le C(M,\eta)(1+a(1)).\end{split}
\end{equation}
If instead~$k_r>k_0$, we employ~\eqref{L23tSx32} to see that
\begin{eqnarray*}&&
a(r)\le\left( \frac1{|B_{\eta^{k_0+C_\star+1}}|}\int_{B_{\eta^{k_0}}}|\nabla u(x)|^p\,dx\right)^{\frac1p}
=\eta^{-n(C_\star+1)/p} a(\eta^{k_0})\\ &&\qquad\qquad\quad\le
\eta^{-n(C_\star+1)/p}\big( C(\eta)M+2^{-k_0}a(1)\big)\le\eta^{-n(C_\star+1)/p}\big( C(\eta)M+a(1)\big)\\&&\quad\qquad\qquad\le
C(M,\eta)(1+a(1)),
\end{eqnarray*}
as long as~$C(M,\eta)$ is chosen large enough.

This and~\eqref{kpd012} give that for all~$r\in\left[\frac{\eta^{k_0+1}}2,\eta\right)$,
we have that~$a(r)\le C(M,\eta)(1+a(1))$. Combining this with~\eqref{ineq-a(r)-r-le-eta^k_0/2},
we deduce that~\eqref{ineq-a(r)-r<1} holds true, as desired.

Now, up to scaling and translations, we can extend~\eqref{ineq-a(r)-r<1} to all balls with center~$x_0$ in~$B_{1/2}$
and sufficiently small radius. Namely, we have that, for all~$r\in(0,\eta]$,
			\begin{equation*}
a(r,x_0)\le C(M,\eta)(1+a(1)),
\end{equation*}
where
$$ 				a(r,x_0):= \left(\fint_{B_{r}(x_0)}\left|\nabla u(x)\right|^p\,dx\right)^{1/p}.$$
			Therefore, according to the Lebesgue Differentiation Theorem, recalling that~$u\in W^{1,p}(B_1),$ we find that, for all~$x_0\in B_{1/2}$,
			\begin{align*}
				&\left|\nabla u(x_0)\right|=\lim_{r\rightarrow 0}a(r,x_0)\le C(M,\eta)(1+a(1))=C\Big(1+\left\|\nabla u\right\|_{L^p(B_1)}\Big),
			\end{align*}
			for some~$C>0$,
			which yields
			\begin{equation}\label{first-claim-theorem-Lipsch-contin}
				\left\|\nabla u\right\|_{L^{\infty}(B_{1/2})}\le C\Big(1+\left\| \nabla u\right\|_{L^p(B_1)}\Big).
			\end{equation}
			So, the first claim in Theorem~\ref{theor-Lipsch-contin-alm-minim} is proved.
			
			We now show that the second claim in Theorem~\ref{theor-Lipsch-contin-alm-minim}
			holds true. For this, we can assume that
			\begin{equation}\label{PROCSH}\{u=0\}\cap B_{s/100}\ne\varnothing\end{equation}
			and we take a Lebesgue point~$\bar{x}\in B_{s/100}$ for~$\nabla u$.
			Up to a translation, we can suppose that~$\bar{x}=0$ and change our assumption~\eqref{PROCSH} into
			\begin{equation}\label{PROCSH2}\{u=0\}\cap B_{s/50}\ne\varnothing.\end{equation}
We claim that, in this situation,
			\begin{equation}\label{alt-c-PK}
			{\mathcal{K}}=\N.
			\end{equation}
			Indeed, suppose not and let~$k_0$ as above (recall~\eqref{pjqsodsckq-=i0dw0pfy9qdhiwcs}).
			Then, in light of~\eqref{second-alternative-dichotomy-B_eta^k},
			we can apply Corollary~\ref{corollary-lemma-second-alternative-dichotomy-improved} (rescaled as before)
			and conclude, by~\eqref{kjhgfdsqwertyuio09876543}, that~$u>0$
			in~$B_{s/2}$, in contradiction with our hypothesis in~\eqref{PROCSH2}.
			This establishes~\eqref{alt-c-PK}.
			
			Therefore, in view of~\eqref{ineq-a(eta^k)}  and~\eqref{alt-c-PK}, for all~$k\in\N$,
$$	a(\eta^k)\le C(\eta)M+2^{-k}a(1),$$
and as a result
$$ |\nabla u(\bar{x})|=|\nabla u(0)|=\lim_{k\to+\infty}a(\eta^k)\le
\lim_{k\to+\infty}\big(C(\eta)M+2^{-k}a(1)\big)=C(\eta)M.$$
Recalling also~\eqref{noticeinpartlsworeteuyti8787686},
the proof of the second claim in Theorem~\ref{theor-Lipsch-contin-alm-minim} is thereby complete.
\end{proof}

\end{document}